\documentclass[11pt]{amsart}
\usepackage[colorlinks=true,pagebackref,hyperindex,citecolor=Green,linkcolor=Blue]{hyperref}
\usepackage[top=1in, bottom=1in, left=1.1in, right=1.1in]{geometry}
\usepackage{amsmath}
\usepackage{amsfonts}
\usepackage{amssymb}
\usepackage{amsthm,amscd}
\usepackage[all]{xy}  
\usepackage{mathrsfs}
\usepackage{enumitem}
\usepackage[T1]{fontenc}
\usepackage[normalem]{ulem}
\usepackage{color}
\usepackage[dvipsnames]{xcolor}

\theoremstyle{theorem}
\newtheorem{theorem}{Theorem}[section]

\newtheorem{corollary}[theorem]{Corollary}
\newtheorem{lemma}[theorem]{Lemma}
\newtheorem{proposition}[theorem]{Proposition}

\theoremstyle{definition}
\newtheorem{definition}[theorem]{Definition}
\newtheorem{question}[theorem]{Question}
\newtheorem{example}[theorem]{Example}

\newtheorem{remark}[theorem]{Remark}

\numberwithin{equation}{subsection}

\newcommand{\m}{\mathfrak{m}}

\newcommand{\cal}{\mathcal}

\newcommand{\Spec}{\operatorname{Spec}}
\newcommand{\Hom}{\operatorname{Hom}}
\newcommand{\Ext}{\operatorname{Ext}}

\newcommand{\Ker}{\operatorname{Ker}}

\newcommand{\IM}{\operatorname{Im}}	
	
\newcommand{\Ann}{\operatorname{Ann}}

\newcommand{\pipn}{\pi^{1/{p^n}}}
\newcommand{\pipinfty}{\pi^{1/{p^\infty}}}
\newcommand{\pip}[1]{\pi^{1/{p^{#1}}}}

\newcommand{\cl}{\operatorname{cl}}
\newcommand{\clfg}{\operatorname{clfg}}
\newcommand{\intr}{\operatorname{tr}}
\newcommand{\taufg}{\tau^{\text{fg}}}

\newcommand{\fg}{finitely-generated}
\newcommand{\CM}{Cohen-Macaulay}

\newcommand{\charp}{characteristic $p>0$}
\newcommand{\echarz}{equal characteristic 0}
\newcommand{\sop}{system of parameters}
\newcommand{\sops}{systems of parameters}
\newcommand{\wrt}{with respect to}
\newcommand{\nzd}{non-zerodivisor}

\definecolor{blue-violet}{rgb}{0.54, 0.17, 0.89}
\definecolor{Blue}{rgb}{0.01, 0.28, 1.0}
\definecolor{gGreen}{rgb}{0.2, 0.8, 0.2}
\definecolor{Green}{rgb}{0.04, 0.85, 0.32}

\begin{document}

\title[Characteristic-free test ideals]{Characteristic-free test ideals}
\author[F.\ P\'erez]{Felipe P\'erez}
\author[Rebecca R.G.]{Rebecca R.G.}
\maketitle


\begin{abstract}
Tight closure test ideals have been central to the classification of singularities in rings of \charp, and via reduction to \charp, in \echarz\ as well. A summary of their properties and applications can be found in \cite{SurveyTestIdeals}.
In this paper, we extend the notion of a test ideal to arbitrary closure operations, particularly those coming from big \CM\ modules and algebras, and prove that it shares key properties of tight closure test ideals. Our main results show how these test ideals can be used to give a characteristic-free classification of singularities, including a few specific results on the mixed characteristic case. 
We also compute examples of these test ideals.
\end{abstract}

\section{Introduction}

The test ideal originated in the study of tight closure \cite{HoHu1}. Since then, it has been used to define a classification of singularities in rings of \charp\ \cite{HoHu1,HoHu2,HoHuStrongFRegularity}, which aligns well with the classification of singularities in \echarz\ \cite{SmithMultiplierTestIdeal,HaraGeometricTestIdeal}. The general idea is that the larger the test ideal, the closer the ring is to being regular, and the smaller the test ideal,  the singular the ring is. The gap in the literature on test ideals is the mixed characteristic case. Recent work of Ma and Schwede \cite{maschwede,MaSchwedeBCM} has begun to fill in this gap, from the perspective of test ideals of pairs. However, most existing results are heavily dependent on the characteristic of the ring, and it is not always known whether corresponding definitions actually agree. In this paper, we study a generalization of the test ideal in a characteristic-free setting. We study test ideals from the perspective of closure operations, mimicking the approach of Hochster and Huneke \cite{fndtc} with regard to the tight closure test ideal but broadening our definition to include test ideals coming from arbitrary closure operations.

We are motivated by work of the second named author on the connections between closure operations given by big \CM\ modules and algebras, and the singularities of the ring \cite{RGRebecca_Closures_Big_CM_modules_and_singularities,bigcmalgebraaxiom}, and encouraged by the fact that these connections hold in all characteristics. More precisely, in \cite{RGRebecca_Closures_Big_CM_modules_and_singularities}, the second named author proved that a ring is regular if and only if all closure operations satisfying certain axioms (Dietz closures) act trivially on modules over the ring. Since big \CM\ modules give Dietz closures, we expect further connections to hold between the singularities of the ring and the big \CM\ module closures over the ring, and we give some of those connections in this paper. In order to do this, we define and study the test ideals given by closures coming from big \CM\ modules and algebras. See Section \ref{testidealsandinteriors} for details.

We prove that the test ideal of a module closure has multiple equivalent definitions, which we use to get our main results connecting singularities to big \CM\ module test ideals.

\begin{theorem}
Let $(R,m,k)$ be a local ring and $E=E_R(k)$ the injective hull of the residue field.
\begin{enumerate}
\item Let cl be a residual closure operation. Then the test ideal $\tau_{\cl}(R)=\text{Ann}\, 0_E^{\cl}$. (Proposition \ref{prop: test = ann})
\item Let $\cl=\cl_B$ be a module closure. If $R$ is complete or $B$ is finitely-presented, then $\tau_{\cl}(R)=\sum_{f \in \Hom_R(B,R)} f(B)$. (Theorem \ref{thm: test = int})
\end{enumerate}
\end{theorem}

In particular, the second result is similar to the result that the tight closure test ideal \[\tau_*(R)=\sum_{e \ge 0} \sum_{\phi \in \Hom_R(R^{1/p^e},R)} \phi((cR)^{1/p^e})\]
for particular elements $c$ \cite{HaraTakagi}. This perspective on the tight closure test ideal is one of the major tools used to study it, as described in \cite{SurveyTestIdeals}. Our second definition also coincides with the trace ideal of the module $B$, as studied in \cite{Lam,LindoHaydeeTraceIdealsAndCentersOfEndomorphismRingsOfModulesOverCommutativeRings}.
By drawing this connection, we open the door for future results on test ideals using the theory of trace ideals, and vice versa. In an upcoming paper with Neil Epstein \cite{epsteinrg}, the second named author has generalized this to a duality between closure operations and interior operations on \fg\ and Artinian modules over complete local rings.

One important consequence of these results is that when the ring is complete and cl is a big \CM\ module closure, $\tau_{\cl}(R)$ is nonzero (Corollary \ref{solidhasnontrivialtestideal}).

We also define a finitistic test ideal of an arbitrary closure operation and discuss cases where it is equal to the (big) test ideal of the same closure operation. In the Gorenstein case, the test ideal of an algebra closure is the whole ring if and only if the corresponding finitistic test ideal is also the whole ring (Proposition \ref{weaklyimpliesstrongly}).

One advantage to working with test ideals of module closures is that, as a consequence of Theorem \ref{thm: test = int}, when the module is \fg, we can compute its test ideal in Macaulay2. This is in contrast to the tight closure test ideal, which is difficult to compute in general. In Section \ref{examples}, we compute examples of test ideals of \fg\ \CM\ modules, and in some cases are able to compute or approximate the ``smallest" \CM\ test ideal.

In summary, our results on the classification of singularities via test ideals are: 

\begin{theorem}
Assume that $R$ is a complete local domain.
\begin{enumerate}
\item $R$ is regular if and only if $\tau_{\cl_B}(R)=R$ for all big \CM\ $R$-modules $B$. (Corollary \ref{testidealtrivialonregular})
\item If $R$ has \charp, then $R$ is weakly F-regular if and only if the finitistic test ideal $\tau_{\cl_B}^{fg}(R)=R$ for all big \CM\ algebras $B$. (Corollary \ref{bigcmalgebratestideals})
\item If $\tau_{\cl_B}(R)=R$ for some big \CM\ module $B$, then $R$ is \CM. (Corollary \ref{onecmtrivialimpliescmtestideal})
\item If $R$ is a \CM\ ring with a canonical module $\omega$, $R$ is Gorenstein if and only if $\tau_{\cl_\omega}(R)=R$. (Corollary \ref{canonicalmodule})
\item If $B$ is a \fg\ \CM\ module, then $V(\tau_{\cl_B}(R)) \subseteq \text{Sing}(R)$. (Corollary \ref{Prop:Singular contains V(tau)})
\item If $R$ has finite \CM\ type but is not regular, then $\tau_{\cl_B}(R)$ is $m$-primary for all \fg\ \CM\ modules $B$. (Proposition \ref{finitecmtypemprimary})
\item If $R$ has countable \CM\ type but is not regular, then $\tau_{\cl_B}(R)$ may not be $m$-primary, even if $B$ is a \fg\ \CM\ module. (Example \ref{whitneyumbrella})
\end{enumerate}
\end{theorem}

 We apply our techniques to the case of mixed characteristic rings in Section \ref{mixedchar}. We propose a mixed characteristic closure operation that satisfies Dietz's axioms (these guarantee that it acts like a big \CM\ module closure--see \cite{DietzGeoffreyD_A_characterization_of_closure_operations_that_induce_BCM_modules,RGRebecca_Closures_Big_CM_modules_and_singularities} for details), and prove that its test ideal can be viewed in three different ways similar to those we gave for module closures earlier. In addition to demonstrating how our results can be used in mixed characteristic, this section shows how our proof techniques can be applied to a broader group of closures than module closures.
 
Throughout the paper, $R$ will denote a commutative Noetherian ring, though some of the $R$-algebras under consideration will not be Noetherian.

\section{Preliminaries} 
\label{preliminaries}

In this section we recall the concepts of closure operations and trace ideals. We record their basic properties for later use and give the appropriate references for their proofs. 

\subsection{Closure Operations}

Given a submodule $N$ of a module $M$, we would like to find a submodule of $M$ containing $N$ that also satisfies some desired properties. This idea is encoded in the following familiar definition. 

\begin{definition} A \emph{closure operation} $\cl$ on a ring $R$ is a map, which to each pair of modules $N\subseteq M$ assigns a submodule $N_M^{\cl} $ of  $M$ satisfying.

\begin{itemize}
\item (Extension) $N\subseteq N_M^{\cl}$,
\item (Idempotence) $(N_M^{\cl})^{\cl}_M=N_M^{\cl}$, and
\item (Order-Preservation) $N_M^{\cl}\subseteq {N'}_M^{\cl}$, for $R$-modules $N\subseteq N'\subseteq M$.
\end{itemize}  

\end{definition}

A particularly  important family of closures are Dietz closures, originally defined in \cite{DietzGeoffreyD_A_characterization_of_closure_operations_that_induce_BCM_modules, dietzbigaxioms}. A local domain has a Dietz closure if and only if it has a big \CM\ module \cite{DietzGeoffreyD_A_characterization_of_closure_operations_that_induce_BCM_modules}.

\begin{definition}
Let $(R,\m)$ be a local domain and $N,M,$ and $W$ be 
$R$-modules with $N\subseteq M$. A closure operation $\cl$ is called a \emph{Dietz closure} if it satisfies the following extra axioms:
\begin{enumerate}

\item (Functoriality) Let $f:M\to W$ be a homomorphism. Then $f(N^{\cl}_M)\subseteq f(N)_W^{\cl}$.
\item (Semi-residuality) If $N^{\cl}_M=N$, then $0^{\cl}_{M/N}=0$.
\item (Faithfulness) The ideal $\m$ is closed in $R$.
\item (Generalized Colon-Capturing) Let $x_1,\ldots,x_{k+1}$ be a partial system of parameters for $R$, and let $J=(x_1,\ldots,x_k)$. Suppose that there exists a surjective homomorphism $f:M\to R/J$ and $v\in M$ such that $f(v)=x_{k+1}+J$. Then $(Rv)^{\cl}_M \cap \ker f \subseteq (Jv)^{\cl}_M$.
\end{enumerate}
Note that these axioms are independent of each other, and an arbitrary closure operation on any ring $R$ can satisfy some subset of them.

\end{definition}

\begin{remark}
The careful reader will note that the axioms, as expressed here, are set in a more general setting than in \cite{DietzGeoffreyD_A_characterization_of_closure_operations_that_induce_BCM_modules}. In  \cite{DietzGeoffreyD_A_characterization_of_closure_operations_that_induce_BCM_modules} the axioms were defined only for complete rings, but this hypothesis was not needed. They were also defined only for finitely-generated modules in \cite{DietzGeoffreyD_A_characterization_of_closure_operations_that_induce_BCM_modules}, but the definitions were later used for arbitrary modules in \cite{dietzbigaxioms}.
\end{remark}

Associated to any $R$-module $B$ we define a closure operation as follows. 

\begin{definition}
Given an $R$-module $B$ (not necessarily \fg), we define a closure operation $\cl_B$ on $R$ by \[u\in N_M^{\cl_B}\text{ if for all }b\in B,\, b\otimes u \in \IM(B\otimes N \to B\otimes M)\] for any pair of $R$-modules $N \subseteq M$ and $u\in M$. This is called a \emph{module closure}.  
\end{definition}

When $B$ is an $R$-algebra, the previous definition can be simplified to $u \in N_M^{\cl_B}$ if and only if \[1\otimes u \in \IM(B\otimes N \to B\otimes M).\]

\begin{remark} \label{rmk: closure for family} We can extend this closure operation to families of modules in certain circumstances. Let $\cal{B}=\{B_i\}_{i\in I}$ be a collection of $R$-modules. We define $\cl_{\cal{B}} = \sum \cl_{B_i}$. This is not in general a closure operation (it is not necessarily idempotent), but since the ring is Noetherian, it can be extended to one by iteration as in \cite[Construction 3.1.5]{Eps}. Alternatively, if the family is directed under generation  (see Definition \ref{generate}), then $\cl_{\cal{B}}$ does form a closure operation. In particular, if the $B_i$ are $R$-algebras that form a directed family, then $\cl_{\cal{B}}$ is a closure operation.

\end{remark}

\begin{definition}[{\cite{hochster1975cbms}}]
Let $(R,m)$ be a local ring. We say that an $R$-module $B$ (not necessarily \fg) is a \textit{big Cohen-Macaulay $R$-module} if $mB \ne B$ and every \sop\ on $R$ is a regular sequence on $B$. Note that these modules are sometimes referred to as balanced big \CM\ $R$-modules (see for example \cite{sharpcmprops}).
\end{definition}

\begin{theorem}[{\cite{DietzGeoffreyD_A_characterization_of_closure_operations_that_induce_BCM_modules}}]
If $B$ is a big \CM\ module, then $\cl_B$ is a Dietz closure.
\end{theorem}

\begin{lemma}[{\cite[Lemma 3.2]{RGRebecca_Closures_Big_CM_modules_and_singularities}}]
\label{residuallemma}
Let $R$ be any ring and $B$ any $R$-module (not necessarily \fg). Then $\cl_B$ satisfies the first two axioms of a Dietz closure, i.e., $\cl_B$ is functorial and semi-residual.
\end{lemma}

\begin{remark} Note that when $M=R$ and $N=I=(f_1,\ldots,f_n) \subseteq R$ is an ideal we have $u\in I^{\cl_B}_R$ if and only if $uB\subseteq IB$. That is, the closure of of an ideal is the collection of all elements that multiply $B$ into $IB$, or equivalently 
\[I^{\cl_B}_R=(IB :_R B).\]
Alternatively, we can write $I^{\cl_B}_R$ as the set of elements $u$ of $R$ for which the equation
\[ ub=f_1X_1+\ldots+f_nX_n \] 
has a solution $(X_1,\ldots,X_n)$ in $B^{\oplus n}$ for every $b\in B$.  Or in the case that $B$ is an $R$-algebra, it is enough to check that 
\[ u 1_B=f_1X_1+\ldots+f_nX_n \] 
has a solution.

We will sometimes write $I^{\cl_B}$ when $R$ is clear from context.
\end{remark}

The following examples show that familiar ideals and closure operations are particular examples of module closures.

\begin{example} Suppose that  $B=R/J$, then we have that $I^{\cl_B}=I+J$.
\end{example}

\begin{example}
If $B=R_f$ for some 
$f\in R$, then for an ideal $I\subseteq R$, $u\in I^{\cl_B}$ if $uR_f\subseteq IR_f$ or equivalently $u\in (I:f^{\infty})$.
\end{example}

\begin{example}
If $R$ is a domain of \charp\ and $B=R^{1/{p^e}}$ for some $e>0$, then for an ideal $I\subseteq R$, $u\in I^{\cl_B}$ if $uR^{1/{p^e}}\subseteq IR^{1/{p^e}}$ or equivalently $u^{p^e}\in I^{[p^e]}$.

If instead $B=R^{1/p^{\infty}}$, then for an ideal $I\subseteq R$, $u\in I^{\cl_B}$ if $uR^{1/{p^{\infty}}}\subseteq IR^{1/{p^{\infty}}}$ which in turn is equivalent to $uR^{1/{p^e}}\subseteq IR^{1/{p^e}}$ for some $e>0$, that is $u^{p^e}\in I^{[p^e]}$ for some $e>0$. This is known as \emph{Frobenius Closure}.
\end{example}

\begin{example}
Suppose that $R$ is an integral domain. The plus closure of $N$ in $M$, denoted $N_M^+$, is the module closure $\cl_{R^+}$, where $R^+$ is the absolute integral closure of $R$ \cite{HochsterHunekeInfiniteIntegralExtensions,SmithParameter,fndtc} (for the extension to modules, see \cite[Remark 7.0.6]{Eps}).
\end{example}

For reference, we list some properties of closure operations and refer the reader to \cite[Lemma 3.1]{RGRebecca_Closures_Big_CM_modules_and_singularities}, \cite[Lemma 1.2]{DietzGeoffreyD_A_characterization_of_closure_operations_that_induce_BCM_modules}, and \cite[Lemma 1.3]{dietzbigaxioms} for the proofs.

\begin{proposition} \label{Closure_properties}
Let $R$ be a ring possessing a closure operation $\cl$. In the following, $N$ and $N'$ are $R$-submodules of the $R$-module $M$, $\mathcal{I}$ is a set, and $N_i \subseteq M_i$ for $i \in \mathcal{I}$ are $R$-modules.

\begin{enumerate}[label=(\alph*)]

\item Suppose that $\cl$ satisfies the functoriality axiom and the semi-residuality axiom. Let $N'\subseteq N\subseteq M$ and $u \in M$. Then $u\in N_M^{\cl}$ if and only if $u+N'\in \left(N/N'\right) _{M/N'}^{\cl}$. 

\item Suppose that 	$\cl$ satisfies the functoriality axiom, $\mathcal{I}$ is any set, $N=\bigoplus_{i\in \mathcal{I}} N_i$, and $M=\bigoplus_{i\in \mathcal{I}} M_i$. Then $N_M^{\cl} = \bigoplus_{i\in \mathcal{I}} (N_i)^{\cl}_{M_i}$.

\item Let $\mathcal{I}$ be any set. If $N_i\subseteq M$ for all $i\in \mathcal{I}$, then $\left( \bigcap_{i \in \mathcal{I}} N_i \right)^{\cl}_M \subseteq \bigcap_{i \in \mathcal{I}} \left(N_i\right)_{M}^{\cl}$.

\item Let $\mathcal{I}$ be any set. If $N_i \subseteq M$ is $\cl$-closed in $M$ for all $i\in \mathcal{I}$, then $\bigcap_{i\in \mathcal{I}}N_i$ is $\cl$-closed in $M$.

\item If $N_1,N_2 \subseteq M$, then $(N_1+N_2)_M^{\cl}=((N_1)_M^{\cl}+(N_2)_M^{\cl})_M^{\cl}$.

\item Suppose that $\cl$ satisfies the functoriality axiom. Let $N\subseteq N' \subseteq M$. Then $N_{N'}^{\cl} \subseteq N_M^{\cl}$.

\item Suppose that $R$ is a domain, $\cl$ satisfies the functoriality axiom, $0_R^{\cl}=0$, and $M$ is a torsion-free \fg\ $R$-module. Then $0_M^{\cl}=0$.

\item Suppose that $(R,\m)$ is local and $\cl$ satisfies the functoriality axiom, the semi-residuality axiom, and the faithfulness axiom. Then, for M a \fg\ $R$-module, and $N\subset M$, $N_M^{\cl}\subseteq N +\m M$. 

\end{enumerate}

\end{proposition}

When the closure operation satisfies the functoriality and semi-residuality axioms,  the elements of the ring multiplying the closure inside the original module can be seen as an annihilator. More precisely:

\begin{lemma}\label{lemma: colon to annihilators}
Let $\cl$ be a closure operation that is functorial and semi-residual. Then for any $R$-module $M$ and any $R$-submodule $N$ of $M$, we have that $\left(N:_R N^{\cl}_M\right)=\Ann_R\left(0^{\cl}_{M/N}\right)$. In particular, this holds for module closures. 
\end{lemma}

\begin{proof}
It is enough to prove that $\left(N:N^{\cl}_M\right)=\left(0:0^{\cl}_{M/N}\right)$. Now part (a) of Proposition \ref{Closure_properties} implies $N^{\cl}_M/N = 0^{\cl}_{M/N}$, from where the result is clear.

\end{proof}

The following proposition gives information about the behavior of module closures under ring extension.

\begin{proposition} 
Let $B,N$ and $M$ be $R$-modules, such that $N\subseteq M$. If $R\to S$ is a ring morphism, then  \[\IM( S \otimes_R N_M^{\cl_{B}} \to S\otimes_R M) \subseteq \IM(S\otimes_R N\to S\otimes_R M)^{\cl_{S\otimes_R B}}_{S\otimes M}. \]
\end{proposition}

\begin{proof}
Suppose that $x \in N_M^{\cl_B}$, then we have that 
\[b \otimes x \in \IM( B \otimes_R N \to B \otimes_R M),\]
for all $b\in B$. Tensoring with $S$ we get  
\[b\otimes s \otimes x \in \IM(B \otimes_R S \otimes_R N \to B \otimes_R S \otimes_R M),\]
for all $b\in B$ and all $s\in S$. But we can rewrite the previous expression as  
\[b\otimes s' \otimes s \otimes x \in \IM(B \otimes_R S \otimes_S S \otimes_R N \to B \otimes_R S \otimes_S S \otimes_R M),\]
for all $b\in B$ and all $s,s'\in S$. Thus $(s\otimes x)\in \IM(S\otimes_R N\to S\otimes_R M)^{\cl_{S\otimes_R B}}_{S\otimes_R M}$.
\end{proof}

\begin{corollary}
Let $B$ be an $R$-module and $\cl_B$ the associated module closure. For any ideal $I$ in $R$ and any prime ideal $P$, \[I^{\cl_B}R_P\subseteq (IR_P)^{\cl_{B_P}}_{R_P}.\] Similarly, if $R$ is a local ring and $\hat{R}$ is its completion at the maximal ideal, then  \[I^{\cl_B}\hat{R} \subseteq (I\hat{R})^{\cl_{\hat{R}\otimes B}}_{\hat{R}}.\]  
\end{corollary}

\begin{definition}[{\cite{Lam}}]
\label{generate}
Recall that a module $B$ is said to \emph{generate} a module $D$ if some direct sum of copies of $B$ maps onto $D$. 
\end{definition}

The generation property enables us to compare the closures given by $B$ and $D$. Before we give the precise result we need a lemma.

\begin{lemma}
Let $R$ be a local ring. If $M$ and $N$ are $R$ modules, then $0^{\cl_M}_{\Hom_R(M,N)}=0$.  
\end{lemma}

\begin{proof}
Let $\phi \in 0^{\cl_M}_{\Hom_R(M,N)}$, then for every $m\in M$ we have that $m \otimes \phi =0 $ in $M \otimes_R \Hom_R(M,N)$. By means of the natural map $M \otimes \Hom_R(M,N) \to N$, given by $m \otimes \phi \mapsto \phi(m)$, we have that $\phi(m)=0$ for all $m\in M$, which implies $\phi =0$. The result follows. 
\end{proof}

The following proposition is the result of a conversation with Yongwei Yao, and gives one case where we have containment of module closures.

\begin{proposition}
\label{generationandclosures}
Let $B$ and $D$ be \fg\ $R$-modules, where $R$ is complete and local. Then $\cl_B \subseteq \cl_D$, i.e. $N_M^{\cl_B} \subseteq N_M^{\cl_D}$ for all $R$-modules $N \subseteq M$, if and only if $B$ generates $D$.
\end{proposition}

\begin{proof}
If $B$ generates $D$ (see Definition \ref{generate}), then $\cl_B \subseteq \cl_D$ by \cite[Proposition 3.6]{RGRebecca_Closures_Big_CM_modules_and_singularities}. For the reverse direction, assume $\cl_B \subseteq \cl_D$.

Let $b_1,\ldots,b_r$ be a generating set for $B$ and $E$ be the injective hull of the residue field of $R$. We have a map 
\[h:\Hom_R(D,E) \to B^{\oplus r} \otimes \Hom_R(D,E)\]
given by $h(f)=(b_1 \oplus \ldots \oplus b_r) \otimes f$. The kernel of this map is the set of elements $f$ of $\Hom_R(D,E)$ such that $b_i \otimes f=0$ for all $1 \le i \le r$, which is equal to the set of $f \in \Hom_R(D,E)$ such that $b \otimes f=0$ for all $b \in B$. This is equal to $0^{\cl_B}_{\Hom_R(D,E)}$. Hence, by our assumption, $0^{\cl_B}_{\Hom_R(D,E)} \subseteq 0^{\cl_D}_{\Hom_R(D,E)}$, but the latter is 0 by the preceding lemma. This implies that $h$ is injective.

Since $h$ is injective, its Matlis dual 
\[h^\vee: \Hom_R(B^{\oplus r} \otimes \Hom_R(D,E),E) \to \Hom(\Hom_R(D,E),E)\]
is surjective. The map $h^\vee$ takes a map $\phi:B^{\oplus r} \otimes \Hom_R(D,E) \to E$ to 
$\phi \circ h:\Hom_R(D,E) \to E$. By Hom-tensor adjointness, we have
\[\begin{aligned}
\Hom_R(B^{\oplus r} \otimes \Hom_R(D,E),E) &\cong \Hom_R(B^{\oplus r},\Hom_R(\Hom_R(D,E),E)) \\
\end{aligned}\]
Under this isomorphism, a map $\psi:B^{\oplus r} \to \Hom_R(\Hom_R(D,E),E)$ is sent to the map $\phi:B^{\oplus r} \otimes \Hom_R(D,E) \to E$ sending
\[(c_1,\ldots,c_r) \otimes f \mapsto (\phi(c_1,\ldots,c_r))(f).\]
Put together, this gives us a surjective map 
\[\Hom(B^{\oplus r},\Hom_R(\Hom_R(D,E),E))) \to \Hom_R(\Hom_R(D,E),E))\]
that sends $\psi:B^{\oplus r}  \to \Hom_R(\Hom_R(D,E),E)$ to $\phi \circ h:\Hom_R(D,E) \to E$. Combining earlier information, $\phi \circ h=\psi(b_1,\ldots,b_r)$.

Since  $R$ is complete and $D$ is \fg, $D \cong \Hom_R(\Hom_R(D,E),E)$, and therefore the map 
\[\Hom(B^{\oplus r},D) \twoheadrightarrow D.\]
given by $\left(\psi:B^{\oplus r}\to D \right) \mapsto \psi(b_1 \oplus \ldots \oplus b_r )$ is surjective. Hence for every $d \in D$, there is a map $B^{\oplus r} \to D$ whose image contains $d$. Therefore, $B$ generates $D$.
\end{proof}

The following proposition characterizes regular rings in terms of the behaviour of Dietz closures. This result describes an important connection between the behavior of big \CM\ module closure operations and the singularities of the ring.

\begin{theorem}[{\cite[Theorem 2]{RGRebecca_Closures_Big_CM_modules_and_singularities}}]  
\label{trivialiffregular}
Suppose that $(R,\m)$ is a local domain that has at least one Dietz closure (in particular $R$ may be any complete local domain). Then $R$ is regular if and only if all Dietz closures on $R$ are trivial on submodules of \fg\ $R$-modules. 
\end{theorem}

Note that this result holds regardless of the characteristic of $R$, as by \cite{HochsterHunekeInfiniteIntegralExtensions, andre}, we know that big \CM\ algebras (and in particular big \CM\ modules) exist over complete local domains of any characteristic.

In fact, the proof of this statement in \cite{RGRebecca_Closures_Big_CM_modules_and_singularities} uses the fact that big \CM\ modules over regular rings are faithfully flat \cite{HochsterHunekeInfiniteIntegralExtensions}, and we get the following corollary to Theorem \ref{trivialiffregular} and its proof in \cite{RGRebecca_Closures_Big_CM_modules_and_singularities}:

\begin{corollary}
\label{rmk regular = closures trivial}
Suppose that $(R,m)$ is a local domain with a big \CM\ module $B$ (in particular, $R$ may be any complete local domain). Then $R$ is regular if and only if all big \CM\ module closures on $R$ are trivial (on submodules of all $R$-modules).
\end{corollary}

\begin{remark} \label{Rmk: Top syz induce nontrivial closures}
Let $(R,\m,k)$ be a Cohen-Macaulay local domain of dimension $d$. If $R$ is approximately Gorenstein (for example if $\dim(R) \ne 1$), then for all $n \ge d$, the $R$-modules $\text{syz}^n(k)$ induce Dietz closures that are trivial 
if and only if $R$ is regular \cite{RGRebecca_Closures_Big_CM_modules_and_singularities}. So when $R$ is not regular, $\text{syz}^n(k)$ gives an example of a nontrivial Dietz closure on $R$.
\end{remark}

We also have the following:

\begin{lemma}
\label{onecmtrivialimpliescmclosure}
Let $R$ be a local domain with a big \CM\ $R$-module $B$ such that $\cl_B$ is trivial on ideals of $R$. Then $R$ is \CM.
\end{lemma}

\begin{proof}
The closure $\cl_B$ captures colons, so for all partial \sops\ $x_1,\ldots,x_{k+1}$ on $R$, we must have
\[(x_1,\ldots,x_k):x_{k+1} \subseteq (x_1,\ldots,x_k)^{\cl_B}_R=(x_1,\ldots,x_k).\]
Hence $R$ is \CM.
\end{proof}

\subsection{Trace ideals and Modules}

\begin{definition}
Let $R$ be a ring and $A,B$ $R$-modules. The \emph{trace of $A$ with respect to $B$} is defined as \[\intr_B(A)=\sum_{\phi:B\to A}\phi(B) \] where the sum runs over all $R$-linear maps from $B$ to $A$. 
\end{definition}

That is, the trace of a module $A$ with respect to another module $B$ is the submodule generated by the images of all possible maps from $B$ to $A$. 

\begin{remark}
\label{mentiontrace}
\begin{enumerate}
\item $B$ generates $A$ if and only if $\intr_B(A)=A$. One example where $B$ generates $A$ is when there is a surjective map from $B$ to $A$, or if $B=R$.
\item  When $A=R$, this is also referred to as the trace ideal, $\text{tr}_B(R)$ \cite{Lam}.
\end{enumerate}
\end{remark}

We collect some basic properties of the trace in the next proposition.

\begin{proposition} [{C.f.  \cite[c.f. Proposition 2.8 ]{LindoHaydeeTraceIdealsAndCentersOfEndomorphismRingsOfModulesOverCommutativeRings} }]
\label{properties interior}

Let $R$ be a ring, and $A,B,C$ $R$-modules. The following holds.

\begin{enumerate}

\item We have \[\intr_B(A)=\IM( \Hom_R(B,A)\otimes B \to A) \] where the map is given by $\phi\otimes b \mapsto \phi(b)$.

\item The behavior with respect to direct sums is given by \[\intr_{B\oplus C}(A)=\intr_{B}(A)+\intr_{C}(A).\] 

\item More generally, if $\{B_i\}_{i\in I}$ is an arbitrary family of $R$-modules, then \[\intr_{\bigoplus_{i\in I} B_i}(A)=\sum_{i \in I} \intr_{B_i}(A).\] 

\item For tensor products, we have \[ \intr_{B\otimes C}(A) \subseteq \intr_B(A) \cap \intr_C(A).\] 

Furthermore, if $B$ generates $\Hom_R(C,A)$ or $C$ generates $\Hom_R(B,A)$, then the equality holds.

\item If $B$ generates $C$ then  \[\intr_B(A)\supseteq \intr_C(A).\]

\item $\intr_A(R)=R$ if and only if $A$ generates all $R$-modules. If $R$ is a local ring then $\intr_A(R)=R$ if and only if $A$ has a free summand \cite[Proposition 2.8, Part iii]{LindoHaydeeTraceIdealsAndCentersOfEndomorphismRingsOfModulesOverCommutativeRings} and \cite[Lemma 3.45]{curtisreiner}.

\item $\intr_{B\otimes \Hom_R(B,A)} (A)=\intr_B(A)$. Furthemore, when $A=R$ and $B$ is reflexive we also have $\intr_{B\otimes \Hom_R(B,R)} (R)=\intr_{\Hom_R(B,R)}(R)$.
\end{enumerate}
\end{proposition}

\begin{proof}
\begin{enumerate}

\item This is clear from the definition.

\item From the definition we see that
\begin{align*}
\intr_{B \oplus C}(A) &= \left( \phi(b,c) \mid \phi\in \Hom_R(B\oplus C,A), b\in B,c \in C \right)\\
&= \left( \phi_1(b)+\phi_2(c) \mid \phi_1\in \Hom_R(B,A), \phi_2\in \Hom_R(C,A), b\in B,c \in C \right)\\
&= \left( \phi_1(b) \mid \phi_1\in \Hom_R(B,A), b\in B\right) + \left( \phi_2(c) \mid \phi_2\in \Hom_R(C,A), c\in C\right)\\
&= \intr_{B}(A)+\intr_C(A)
\end{align*}

\item We proceed as in the previous case
\begin{align*}
\intr_{\bigoplus_{i\in I} B_i}(A)&= \left( \phi((b_i)_{i\in I}) \mid \phi\in \Hom_R\left(\oplus_{i\in I} B_i,A\right), b_i\in B_i, b_i = 0 \text{ for all but finitely many } i \right)\\
&=\left( \sum_{i\in I} \phi_i(b_i) \mid \phi_i\in \Hom_R\left( B_i,A\right), b_i\in B_i, b_i = 0 \text{ for all but finitely many } i \right)\\
&=\sum_{i\in I}  \left( \phi_i(b_i) \mid \phi_i\in \Hom_R(B_i,A), b_i\in B_i,\right) \\
&= \sum_{i\in I}\intr_{B_i}(A)
\end{align*}

which is what we wanted.  

\item Note that for any $\phi \in \Hom_R(B\otimes C,A)$ and $c \in C$, we have a map $\phi(-\otimes c):B\to A$ sending $b \mapsto \phi(b \otimes c)$. Hence $\phi(b\otimes c) \in \intr_B(A)$ for all $b \in B$. Similarly, $\phi(b\otimes c) \in \intr_C(R)$ and the result follows. 

To get the equality, assume that $B$ generates $\Hom_R(C,A)$. Then for $a\in \intr_B(A)\cap \intr_C(A)$ there exists $\phi:C\to A$ such that $\phi(c)=a$ for some $c\in C$. Now as $B$ generates $\Hom_R(C,A)$, there exists a map $\Psi:B \to \Hom_R(C,A)$ and an element $b \in B$ such that $\Psi(b)=\phi$. Consider the map $B\otimes C \to A$ given by $y\otimes z \mapsto \Psi(y)(z)$. This map is well defined and $b\otimes c \mapsto a$. The result follows. The case where $C$ generates $\Hom_R(B,A)$ works the same way.  

\item This follows from the fact that every element $a$ in $\intr_C(A)$ can be obtained via a map $C\to A$ and an element $c \in C$. This element $c$ will be in the image of some map $B \to C$, and so its image $a$ in $A$ can be obtained via the composition $B\to C \to A$. Hence $a \in \intr_B(A)$.

\item Follows as in the references, where the hypothesis that $A$ is \fg\ used by Lindo is not needed.

\item By part (4) we have that $\intr_{B\otimes \Hom_R(B,A)}(A) \subseteq \intr_B(A)$. On the other hand we have the map $B \otimes \Hom_R(B,A) \to A$ given by $b\otimes \phi \mapsto \phi(b)$. This implies $\intr_{B\otimes \Hom_R(B,A)}(A) \supseteq \IM(B \otimes \Hom_R(B,A) \to A) = \intr_B(A)$. The last assertion is trivial after noting that $B=\Hom_R(\Hom_R(B,R),R)$.

\end{enumerate}
\end{proof}

The result below relates traces of modules in an exact sequence.

\begin{proposition}
Let $0\to B \xrightarrow{\alpha} C \to D \to 0$ be a short exact sequence of $R$-modules, and $A$ any other $R$-module. If $J= \Ann_R(\Ext^1_R(D,A))$, then \[ J\intr_B(A)+\intr_D(A)\subseteq\intr_C(A).\]
\end{proposition}

\begin{proof}
By Proposition \ref{properties interior} part 5 we have that $\intr_D(A) \subseteq \intr_C(A)$. Let $a \in \intr_B(A)$. Then there exist $\phi \in \Hom_R(B,A)$ and $b\in B$ such that $\phi(b)=a$. From the exact sequence 
\[ \Hom_R(C,A) \to \Hom_R(B,A) \to \Ext^1_R(D,A) \]
we can conclude that for any $r \in \Ann_R(\Ext^1_R(D,A)),$ $r\phi \in \IM(\Hom_R(C,A)\xrightarrow{\hat{\alpha}} \Hom_R(B,A))$, say $r\phi=\hat{\alpha}(\tilde{\phi})=\tilde{\phi} \circ \alpha$. This implies that $ra = r \phi(b) = (\tilde{\phi}\circ \alpha) (b)$. Setting $c=\alpha(b)$, we have $ra=\tilde{\phi}(c)$. The result follows. 

\end{proof}

\section{Test ideals and Trace Ideals}
\label{testidealsandinteriors}

In this section we define the test ideal of an arbitrary closure operation, give some of its basic properties, and prove that the test ideal of a module closure is a trace ideal.

\begin{definition}
Let $R$ be a ring and $\cl$ be a closure operation on $R$-modules. The \emph{big test ideal of $R$ associated to $\cl$} is defined as \[ \tau_{\cl}(R)=\bigcap_{N\subseteq M}\left(N:N^{\cl}_M\right)\] where the intersection runs over any (not necessarily finitely-generated) $R$-modules $N,M$. In the case that $\cl$ is generated from a $R$-module $B$, (resp. a family $\mathcal{B}$) that is $\cl =\cl_B$ we also denote this ideal by $\tau_B(R)$ (resp. $\tau_{\mathcal{B}}(R)$). We sometimes refer to the big test ideal as the test ideal.

Similarly, we define the \emph{finitistic test ideal of $R$ associated to $\cl$} as 
\[\taufg_{\cl}(R)=\bigcap_{\begin{subarray}{c}{N \subseteq M}\\
    \text{$M/N$ f.g.}\end{subarray}} \left(N:N_M^{\cl}\right).\]
In the case where $\cl=\cl_B$ for some $R$-module $B$, we denote this ideal by $\taufg_B(R)$.

Note that the big test ideal is always contained in the finitistic test ideal.
\end{definition}

When cl is tight closure, these definitions agree with the the tight closure test ideal as given in \cite[Definition 8.22]{HoHu1}.
As an immediate consequence we get:

\begin{corollary} \label{cor: test ideal = R iff closure is trivial}
Let $\cl$ be a closure operation. Then, the test ideal $\tau_{\cl}(R)$ is equal to $R$ if and only if for every inclusion of $R$-modules $N\subseteq M$, we have $N^{\cl}_M=N$.

Similarly, $\taufg_{\cl}(R)=R$ if and only if for every inclusion of  $R$-modules $N \subseteq M$, with $M/N$ \fg, we have $N^{\cl}_M=N$.
\end{corollary}

\begin{lemma}
\label{testidealall0}
Let $\cl$ be a closure that is functorial and semi-residual. Then 
\[\tau_{\cl}(R)=\bigcap_{M\text{ an $R$-module}} \Ann_R\left( 0^{\cl}_{M}\right).\]
Additionally,
\[\taufg_{\cl}(R)=\bigcap_{M\text{ a f.g. $R$-module}} \Ann_R\left( 0^{\cl}_{M}\right).\]
\end{lemma}

\begin{proof}
Lemma \ref{lemma: colon to annihilators} implies that for any $R$-modules $N \subseteq M$, $\left(N:N^{\cl}_M\right)=\Ann_R\left(0^{\cl}_{M/N}\right)$, and so \[\tau_{\cl}(R)=\bigcap_{M\text{ an $R$-module}} \Ann_R\left( 0^{\cl}_{M}\right).\] 
The second result follows as the intersection will be over all $M/N$ \fg.
\end{proof}

\begin{remark}
The finitistic test ideal could be taken as the intersection over all $R$-modules $N \subseteq M$ where $M$ is \fg. If cl is functorial and semi-residual, then by the proof of Lemma \ref{testidealall0}, this is equal to 
\[\bigcap_{M\text{ a f.g. $R$-module}} \Ann_R\left( 0^{\cl}_{M}\right),\]
and so it is equal to our definition of the finitistic test ideal. In particular, this holds for module closures.
\end{remark}

\begin{corollary}
\label{testidealtrivialonregular}
If $R$ is a regular local ring, and cl is a Dietz closure on $R$, then $\taufg_{\cl}(R)=R$. In this case, if $\cl=\cl_B$ for some big \CM\ module $B$, then $\tau_{\cl}(R)=R$ as well. In fact, if $R$ is a complete local domain, $R$ is regular if and only if $\tau_B(R)=R$ for all big \CM\ modules $B$.
\end{corollary}

\begin{proof}
The first claim follows from the definition of a test ideal, Theorem \ref{trivialiffregular}, and Lemma \ref{testidealall0}: if $R$ is regular and cl is a Dietz closure, cl is trivial on \fg\ $R$-modules, so $\taufg_{\cl}(R)=R$. 
By Corollary \ref{rmk regular = closures trivial}, $R$ is regular if and only if $\cl_B$ is trivial for all big \CM\ modules $B$. The result follows from Corollary \ref{cor: test ideal = R iff closure is trivial}.
\end{proof}

\begin{corollary}
\label{onecmtrivialimpliescmtestideal}
Let $R$ be a local domain with a big \CM\ module $B$ such that $\tau_B(R)=R$ (or $\tau_B^{fg}(R)=R$). Then $R$ is \CM.
\end{corollary}

\begin{proof}
By Corollary \ref{cor: test ideal = R iff closure is trivial}, $\cl_B$ is trivial on ideals of $R$. Hence $R$ is \CM\ by Lemma \ref{onecmtrivialimpliescmclosure}.
\end{proof}

Note that if $R$ is \CM, then $\tau_R(R)=R$, so the converse holds.

 It follows from the definition that $\tau_{\cl}(R) \subseteq \taufg_{\cl}(R)$, leading to the following question that is still open in most cases for the tight closure test ideal.
 
 \begin{question} \label{Big=Small?}
Do the big test ideal and the finitistic test ideal coincide? More specifically, what are the conditions needed on a ring $R$ or on a closure operation $\cl$ so that $\tau_{\cl}(R)=\taufg_{\cl}(R)$? 
\end{question}

The following result answers this question in one special case. We will be able to say more once we prove Proposition \ref{prop: test = ann}, our first result giving an alternate definition of the test ideal.

\begin{proposition}
Let $\cal{B}$ be a directed family of flat $R$-algebras, or a single flat $R$-module $B$. Then $\tau_{\cal{B}}(R)=\taufg_{\cal{B}}(R)$.
\end{proposition}

\begin{proof}
Let $\cl=\cl_{\cal{B}}$, and $\clfg$ denote 
the closure given by:
\[u \in N_M^{\clfg} \text{ if for some } N \subseteq M_0 \subseteq M, M_0/N \text{ \fg, } u \in N_{M_0}^{\cl}.\]

 We claim that $\clfg = \cl$. To see that $\clfg \subseteq \cl$, note that by part (f) of Proposition \ref{Closure_properties}, for any $N \subseteq M_0 \subseteq M$,
 \[N_{M_0}^{\cl} \subseteq N_M^{\cl}.\]

For the other inclusion, suppose that $u \in N_M^{\cl}$. Then there is some $B \in \cal{B}$ such that $u \in N_M^{\cl_B}$. Since $B$ is a flat $R$-module, by \cite[Proposition III.12]{RGthesis},
$\cl_B$ is hereditary, i.e. for any $N \subseteq M_0 \subseteq M$, $N_M^{\cl_B} \cap M_0=N_{M_0}^{\cl_B}$. We have
\[\bigcup_{\begin{subarray}{c}{N \subseteq M_0 \subseteq M}\\
    \text{$M_0/N$ f.g.}\end{subarray}} N_{M_0}^{\cl_B}
    = \bigcup_{\begin{subarray}{c}{N \subseteq M_0 \subseteq M}\\
    \text{$M_0/N$ f.g.}\end{subarray}} \left(N_M^{\cl_B} \cap M_0\right)
    =N_M^{\cl_B} \bigcap \left( \bigcup_{\begin{subarray}{c}{N \subseteq M_0 \subseteq M}\\
    \text{$M_0/N$ f.g.}\end{subarray}} M_0\right).
    \]
Note that $M/N$ is the union of its \fg\ submodules, so $M$ can be written as the union of the $M_0$ above. Hence the final step is equal to $N_M^{\cl_B} \cap M=N_M^{\cl_B}.$ This implies that $u \in N_{M_0}^{\cl_B}$ for some $N \subseteq M_0 \subseteq M$ with $M_0/N$ \fg. Hence $u \in N_{M}^{\clfg}$.

Now we show that $\tau_{\cal{B}}(R)=\taufg_{\cal{B}}(R)$. The forward inclusion always holds. For the reverse inclusion, suppose that $u \in \taufg_{\cal{B}}(R)$. We would like to show that for arbitrary $R$-modules $N \subseteq M$, $uN_M^{\cl} \subseteq N$.
Since $N_M^{\cl}=N_M^{\clfg}$, for every $v \in N_M^{\cl}$, there is some $N \subseteq M_0 \subseteq M$ with $M_0/N$ \fg\ such that $v \in N_{M_0}^{\cl}$. Hence $uv \in N$. This implies that $u \in \tau_{\cal{B}}(R)$, which gives us the result.
\end{proof}

\begin{proposition} \label{prop: test = ann}
Let $\cl$ be a closure on a local ring $(R,\m,k)$ satisfying the first two Dietz axioms, functoriality and semi-residuality, and $E=E_R(k)$ be the injective hull of the residue field $k$. Let $\tau_{\cl}(R)$ denote the big test ideal associated to $\cl$. Then $\tau_{\cl}(R)=\Ann_R \left(0^{\cl}_{E}\right)$. Additionally, $\taufg_{\cl}(R)$ is the annihilator of 
\[0^{\clfg}_E=\{u \in E \mid \text{ for some \fg\ } E' \subseteq E, u \in 0^{\cl}_{E'}\}.\]
\end{proposition}

\begin{proof}{C.f. \cite[Proposition 8.23]{HoHu1}.} 
By Lemma \ref{testidealall0}, 
\[\tau_{\cl}(R)=\bigcap_{M\text{ an $R$-module}} \Ann_R\left( 0^{\cl}_{M}\right).\]
We now show that $\bigcap_M \Ann_R 0^{\cl}_M=\Ann_R 0^{\cl}_{E}$. That the first is contained in the second is clear. For the other inclusion let $u\in R-\{0\}$ such that $u0^{\cl}_{E}=0$, and let $M$ be an $R$-module such that $u0^{\cl}_{M} \ne 0$. Then there is some $x \in 0^{\cl}_{M} \subseteq M$ such that $ux \ne 0$ in $M$. Choose $N \subseteq M$ maximal with respect to not containing $ux$. Replace $M$ by $M/N$ and $x$ by $\bar{x}$. By \cite[Lecture of September 17]{fndtc}, every \fg\ submodule of $M$ has finite length and $ux$ spans its socle. Hence $ux$ spans the socle of $M$, and so $Rux\cong kux \cong k$, and $M$ is an essential extension of this copy of $k$. 
Hence we can embed $M$ in $E$, and so by part (f) of Proposition \ref{Closure_properties},
 \[ u\cdot 0^{\cl}_{M} \subseteq u\cdot 0^{\cl}_{E}=0, \] which contradicts our choice of $x$. The result follows. 

Now we show that $\taufg_{\cl}(R)=\Ann_R 0^{\clfg}_E$. We have 
\[\taufg_{\cl}(R)= \bigcap_{M \text{ f.g.}} \Ann_R(0^{\cl}_M).\]
To see that this is contained in $\Ann_R 0^{\clfg}_E$, notice that every element $v \in 0^{\clfg}_E$ is contained in $0^{\cl}_{E'}$ for some \fg\ $E' \subseteq E$. So an element $u \in R$ that kills $0^{\cl}_M$ for every \fg\ $R$-module $M$ will kill $v$. Hence $\taufg_{\cl}(R) \subseteq \Ann_R 0^{\clfg}_E$. For the reverse inclusion, let $u \in R-\{0\}$ such that $u 0^{\clfg}_E=0$, and let $M$ be a \fg\ $R$-module such that $u0_M^{\cl} \ne 0$. The rest of the argument follows as for the non-\fg\ case, with the addition to the last line that since $M$ is \fg, $u0_M^{\cl} \subseteq u0_E^{\clfg}=0$.
\end{proof}

Using this alternative description of the test ideal, we give an additional partial answer to Question \ref{Big=Small?}. This result is the module-closure version of Theorem 3.1 of \cite{HoHuStrongFRegularity} or the notes of October 22nd and 24th of \cite{fndtc}.

\begin{proposition}
\label{weaklyimpliesstrongly}
Let $R$ be a Gorenstein local ring, and $B$ any $R$-algebra or \fg\ $R$-module. Then $\taufg_B(R)=R$ if and only if $\tau_B(R)=R$.
\end{proposition}

\begin{proof}
We always have $\tau_B(R) \subseteq \taufg_B(R)$, so the reverse direction holds without the Gorenstein assumption on $R$. For the forward statement,  denote $\cl_B$ by $\cl$, and suppose that $\taufg_B(R)=R$. Then, $I^{\cl}_R=I$ for all ideals $I$ of $R$.

Let $x_1,\ldots,x_d$ be a \sop\ on $R$, and $I_t=(x_1^t,x_2^t,\ldots,x_d^t)$. Since $R$ is Gorenstein local, we have $E_R(k)=\varinjlim_t R/I_t$, where the maps $R/I_t \to R/I_{t+1}$ are given by multiplication by $y=x_1 \cdots x_d$. 
Using the notation of \cite[Lecture of October 24th]{fndtc}, let us denote the equivalence class of an element of $R$ under the composition $R \twoheadrightarrow R/I_t \hookrightarrow E$ by $(u;I_t)$. So $(u;I_t)=(uy^r;I_{t+r})$. 
Suppose that some element $v=(u;I_t) \in E$ is in $0^{\cl}_E$. Let $\{b_1,\ldots,b_n\}$ be a set of generators for $B$ if $B$ is a module, or $\{1\}$ if $B$ is an $R$-algebra. Then for $1 \le i \le n$, $b_i \otimes v=0$ in $B \otimes E$. This holds if and only if for each $i$, there is some $r_i$ such that $b_i \otimes uy^{r_i+s}=0$ in $B \otimes R/I_{t+r_i+s}$ for all $s \ge 0$. Set $r=\text{max}_i \{r_i\}$. Identifying $B \otimes R/I_t$ with $B/I_tB$, this implies that $uy^rb_i \in I_{t+r}B$ for each $i$. But this is exactly equivalent to $uy^r \in (I_{t+r})^{\cl}_R$. Since by assumption $(I_{t+r})^{\cl}_R=I_{t+r}$, we have $uy^r \in I_{t+r}$ for sufficiently large values of $r$. This implies that $v=0$ in $E$. Hence $0^{\cl}_E=0$, and thus $\tau_B(R)=R$.
\end{proof}

We can use the previous result to give a similar result for families.

\begin{corollary}
\label{gorensteinfinitisticandbigcoincide}
Let $R$ be a Gorenstein local ring and $\cal{B}$ a directed family of $R$-algebras or a family of  \fg\  $R$-modules directed under generation. Then $\taufg_\cal{B}(R)=R$ if and only if $\tau_{\cal{B}}(R)=R$. 
\end{corollary}

\begin{proof}
Let $\cl=\cl_{\cal{B}}$.  The piece we need to prove is that if $\taufg_{\cl}(R)=R$, then $\tau_{\cl}(R)=R$. Suppose that $\taufg_{\cl}(R)=R$.
Let $v \in 0^{\cl}_E$. Then there is some $B \in \cal{B}$ such that $v \in 0^{\cl_B}_E$. For every $B \in \cal{B}$, $\cl_B \subseteq \cl_{\cal{B}}$. Since $\taufg_{\cl}(R)=R$, $\taufg_{\cl_B}(R)=R$. By Proposition \ref{weaklyimpliesstrongly}, this implies that $\tau_{\cl_B}(R)=R$. Hence by Proposition \ref{prop: test = ann} $0^{\cl_B}_E=0$, whch implies that $v=0$. Therefore, $0^{\cl}_E=0$, and so $\tau_{\cl}(R)=R$.
\end{proof}

The following theorem connects test ideals with trace ideals, and is the key component of many of our results. This connects the idea of the test ideal with representation theoretic ideas.

\begin{theorem} \label{thm: test = int}
Let $R$ be local and $\cl=\cl_B$ for some $R$-module $B$. If $B$ is a finitely presented $R$-module or $R$ is complete then \[ \tau_{\cl}(R)=\intr_{B}(R)\]
\end{theorem}

\begin{proof}

Let $E=E_R(k)$ be the injective hull of the residue field $k$ of $R$. By Proposition \ref{prop: test = ann}, $\tau_{\cl}(R)=\Ann_R(0^{\cl}_E)=(0:0^{\cl}_E)$; hence $c\in \tau_{\cl}(R)$ if and only if $c\cdot 0^{\cl}_E=0$, but 
\[0^{\cl}_E=\bigcap_{b \in B} \ker(E\to B\otimes E),\] 
where the map $E \to B \otimes E$ corresponding to $b \in B$ is given by $e \mapsto b \otimes e$.
Since $E$ is Artinian, there are elements $b_1,\ldots,b_n \in B$ such that this is equal to 
\[\bigcap_{b\in \{b_1,\ldots,b_n\}} \ker(E \to B \otimes E).\] 
We can rewrite this as $\ker(\phi)$, where $\phi=(\phi_1,\ldots,\phi_n):E \to (B \otimes E)^{\oplus n}$ sends 
\[e \mapsto (b_1 \otimes e,b_2 \otimes e,\ldots,b_n \otimes e).\] 

First, suppose that $c \in \tau_{\cl}(R)$, so that $c \cdot 0_E^{\cl}=c\ker(\phi)=0$. Then \[0^{\cl}_E\subseteq \Ann_E(c),\] and by Matlis duality the map  \[\frac{\hat{R}}{c\hat{R}}=\Hom_{\hat{R}}(\Ann_E(c),E) \to \Hom_{\hat{R}}(0^{\cl}_E,E)\] is surjective. But applying $\Hom_R(\_,E)$ to the exact sequence \[0\to \ker(\phi)  \to E \xrightarrow{\phi} (B\otimes E)^{\oplus n} \] gives \[ \Hom_{\hat{R}}(0^{\cl}_E,E)=\frac{\hat{R}}{\sum_{i=1}^n \IM(\Hom_{\hat{R}}(B\otimes \hat{R},\hat{R})\to \hat{R})},\] where $i$th map $\Hom_{\hat{R}}(B\otimes \hat{R},\hat{R})\to \hat{R}$ is given by $\phi\mapsto \phi(b_i)$. From the surjection $\hat{R}/c\hat{R}\to \Hom_{\hat{R}}(0_E^{\cl},E)$ we can now conclude that 
\[c \hat{R} \subseteq \sum_{i=1}^n \IM(\Hom_{\hat{R}}(B\otimes \hat{R},{\hat{R}})\to \hat{R}) .\] 

In the complete case, the denominator is contained in $\intr_{B}(R)$, so this implies that $cR \subseteq \intr_{B}(R)$. In the case that $B$ is finitely presented, since $\Hom$ commutes with flat base change, the last expression is equal to 

\[ \left(\sum_{i=1}^n \IM(\Hom_{R}(B,{R})\to R) \right) \otimes \hat{R} \]

It then follows by the faithful flatness of completion that 
\[c \in \sum_{i=1}^n \IM(\Hom_{R}(B,{R})\to R) \subseteq \intr_{B}(R).\]

For the reverse containment, suppose that $c \in \intr_{B}(R)$. Then there are $b'_1,\ldots,b'_m$ such that 
\[c \in \sum_{i=1}^m \IM(\Hom_R(B,R) \to R),\]
 where the $i$th map $\Hom_R(B,R) \to R$ sends $f \mapsto f(b'_i)$. We can enlarge the set $b_1,\ldots,b_n$ from the setup to include $b'_1,\ldots,b'_m$. Then 
 \[c \in \sum_{i=1}^n \IM(\Hom_R(B,R) \to R).\] Hence we have a surjection 
 \[R/cR \to \frac{R}{\sum_{i=1}^n \IM(\Hom_R(B,R))}.\]
  Applying $\Hom_R(\_,E)$, we get an injection
\[ \Hom_R\left(\frac{R}{\sum_{i=1}^n \IM(\Hom_R(B,R) \to R)},E\right) \hookrightarrow \Hom_R(R/cR,E)=\Ann_E c.\] But the module on the left is $0_E^{\cl}$. Hence $c \in \Ann_R 0_E^{\cl}$, which is equal to $\tau_{\cl}(R)$.

\end{proof}

\begin{remark}
The second direction of the previous theorem works in greater generality; in particular it shows that for any local ring $R$ (not necessarily complete) and any $R$-module $B$ (not necessarily \fg) we have \[\intr_B(R)\subseteq \tau_B(R).\]
\end{remark}

\begin{remark}
\label{noncompletecounterexample}
The following example shows that when $R$ is not complete and $B$ is not finitely presented the trace ideal may differ from the test ideal.

We start with \cite[Example 4.5.1]{DattaSmithFrobeniusAndValuationRings} which allows us to build a DVR $V$ whose fraction field is $\mathbb{F}_p(x,y)$. In this case $V$ is a Noetherian, regular ring of dimension 1, which is not $F$-finite. By \cite[Lemma 2.4.2]{DattaSmithFrobeniusAndValuationRings} this implies that $\Hom_V(V^{1/p},V)=0$, hence we have $\intr_{V^{1/p}}(V)=0$. 
On the other hand, as $V$ is a regular ring of dimension one, it is a domain. Hence  $V^{1/p}$ is torsion-free. Additionally, $mV^{1/p}\neq V^{1/p}$, so $V^{1/p}$ is a \CM\ module. This implies that $\tau_{V^{1/p}}(V)=V \ne 0$.
[Note: The paper as originally published has an error, which the authors corrected in an erratum, but the example and the lemma we are using are correct.]

\end{remark}

The following results use Theorem \ref{thm: test = int} to extend our knowledge of test ideals and closure operations, and in particular give an important case when the test ideal is nonzero. First we recall a definition:

\begin{definition}[{\cite{solidclosure}}]
Let $R$ be a domain. An $R$-module $B$ is solid if $\Hom_R(B,R)$ is nonzero.
\end{definition}

\begin{corollary}
\label{solidhasnontrivialtestideal}
If $R$ is local, $\cl=\cl_B$ for some solid $R$-module $B$, and either $R$ is complete or $B$ is \fg,
then we have $\tau_{\cl} (R)\neq 0$. 
Consequently, $\taufg_{\cl}(R) \ne 0$ as well.

In particular, if $R$ is a complete local domain and $B$ is a big \CM\ $R$-module, then $\tau_{B}(R) \ne 0$.
\end{corollary}

\begin{proof}
Assume that $\cl=\cl_B$ for some solid $R$-module $B$. Since $\tau_{\cl}(R)=\intr_{B}(R)$, and there is a nonzero map $B \to R$, $\tau_{\cl}(R) \ne 0$.

If $R$ is a complete local domain, then $B$ is solid \cite[Lecture of September 7th]{fndtc}, and the last statement follows.
\end{proof}

\begin{corollary}
Let cl be a Dietz closure and $R$ a complete local domain. Then $\taufg_{\cl}(R) \neq 0$.
\end{corollary}

\begin{proof}
By \cite{RGRebecca_Closures_Big_CM_modules_and_singularities}, there is a big \CM\ module $B$ such that for all \fg\ $R$-modules $N \subseteq M$, $N_M^{\cl} \subseteq N_M^{\cl_B}$.

Since $B$ is solid over $R$, $\taufg_{\cl_B}(R) \ne 0$. Since $N_M^{\cl} \subseteq N_M^{\cl_B}$ for \fg\ $R$-modules $N \subseteq M$, $\taufg_{\cl}(R) \supseteq \taufg_{B}(R)$, so $\taufg_{\cl}(R)$ is nonzero as well.
\end{proof}

\begin{corollary} \label{cor: test ideal trivial = direct summand}
Let $R$ be local, $S$ an $R$-module, and either $R$ is complete or $S$ is \fg. Then $\tau_S(R)=R$ if and only if $S$ has a free summand, and consequently, $\cl_S$ is trivial if and only if $S$ has a free summand.
\end{corollary}

\begin{proof}
By part 6 of Proposition \ref{properties interior}, $\intr_S(R)=R$ if and only if $S$ has a free summand. Additionally, by Theorem \ref{thm: test = int}, $\tau_S(R)=\intr_S(R)$, and by Corollary \ref{cor: test ideal = R iff closure is trivial}, $\tau_S(R)=R$ if and only if $\cl_S$ is trivial.
\end{proof}

When $R$ is local has a canonical module $\omega$, $\omega$ has a free summand if and only if $R$ is Gorenstein, and hence
$\intr_\omega(R)$ can be used to detect whether the ring is Gorenstein \cite[Lemma 2.1]{herzoghibistamate}). We give a test ideal interpretation of this result.

\begin{corollary}
\label{canonicalmodule}
Let $R$ be a reduced (or generically Gorenstein) \CM\ local ring with a canonical module $\omega$. Then $R$ is Gorenstein if and only if $\tau_{\omega}(R)=R$.
\end{corollary}

\begin{proof}
By \cite[Lemma 2.1]{herzoghibistamate}, $R$ is Gorenstein if and only if $\intr_{\omega}(R)=R$.
The result now follows from Theorem \ref{thm: test = int}.
\end{proof}

\begin{corollary}
Let $A$ and $B$ be $R$-modules satisfying the conditions of the theorem. If $\cl_A$ and $\cl_B$ are the closure operations associated to $A$ and $B$, then \[\tau_{A\oplus B}(R) = \tau_A(R)+\tau_B(R).\]
\end{corollary}

\begin{proof}
This follows from the previous Theorem and Proposition \ref{properties interior} part (2).
\end{proof}

\section{Test ideals of Families}
\label{families}

We extend the concept of test ideal introduced in the previous setting to that of families of modules. We can make this definition even when the family of modules does not give an idempotent closure operation, which is one way to deal with the question of how large the sum of the corresponding module closure operations can be (discussed in \cite[Section 9.2]{RGRebecca_Closures_Big_CM_modules_and_singularities}).  We will then discuss the test ideals of specific families of big \CM\ modules and algebras and connect them to the singularities of the ring.

\begin{definition}
\label{familytestideal}
Let $\mathcal{B}$ be a family of $R$-modules, not necessarily \fg. We define the test ideal associated to $\mathcal{B}$ as 
\[ \tau_{\mathcal{B}}(R):= \bigcap_{B \in \mathcal{B}} \tau_B(R) \]

\end{definition}

We list an immediate set of properties

\begin{lemma}
Let $R$ be a commutative ring and $\mathcal{B},\mathcal{C}$ families of $R$-modules, then 
\begin{enumerate}[label=(\alph*)]
\item $\tau_{\emptyset}(R)=R$.
\item If $(0)\in \mathcal{B}$ then $\tau_{\mathcal{B}}(R)=0$, in particular $\tau_{R-mod}(R)=0$.
\item If $\mathcal{B}\subseteq \mathcal{C}$, then $\tau_{\mathcal{B}}(R) \supseteq \tau_{\mathcal{C}}(R)$.
\item $\tau_{\mathcal{B}\cup \mathcal{C}}(R)=\tau_{\mathcal{B}}(R) \cap \tau_{\mathcal{C}}(R)$.
\item $\tau_{\mathcal{B}\cap \mathcal{C}}(R)\supseteq \tau_{\mathcal{B}}(R) + \tau_{\mathcal{C}}(R)$.
\end{enumerate}
\end{lemma}

Note that if $\mathcal{B}$ is a directed family of $R$-algebras or of $R$-modules directed under generation, so that it defines a closure operation, then this definition of the test ideal agrees with our prior definition:
 
\begin{proposition}
\label{twotestidealsareequal}
Suppose that $\mathcal{B}$ is a directed family of $R$-algebras, or of $R$-modules directed under generation. Let $\cl$ be the closure operation associated to $\mathcal{B}$. Then
$\tau_{\mathcal{B}}(R)=\tau_{\cl}(R)$.
\end{proposition}

\begin{proof}
We have 
\[
\begin{aligned}
\tau_{\cl}(R) &=\bigcap_{M \text{ an $R$-module}} \Ann_R 0_M^{\cl} \\
&= \bigcap_M \Ann_R \left(\sum_{B \in \mathcal{B}} 0_M^{\cl_B}\right) \\
&=\bigcap_{M} \bigcap_{B \in \mathcal{B}} \Ann_R 0_M^{\cl_B} \\
&=\bigcap_{B \in \mathcal{B}} \tau_B(R) \\
&=\tau_{\mathcal{B}}(R).
\end{aligned}
\]
\end{proof}

\begin{corollary}
\label{testidealoffamilyintermsofint}
Under the conditions of Theorem \ref{thm: test = int} (i.e., $R$ is complete local, or $R$ is local and every $B \in \mathcal{B}$ is finitely-presented),
 \[ \tau_{\mathcal{B}}(R)= \bigcap_{B\in \mathcal{B}}\intr_B(R)\]
 \end{corollary}

\begin{corollary} 
Let $R$ be a complete local domain. If $\cal{S}$ is a directed family of $R$-algebras or a family of $R$-modules directed under generation (so that $\cl_{\cal{S}}$ is a closure operation), then $\cl_{\cal{S}}$ is trivial if and only if for every $S \in \cal{S}$, $S$ has a free summand.
\end{corollary}

\begin{proof}
By Proposition \ref{twotestidealsareequal} and Corollary \ref{testidealoffamilyintermsofint}, 
\[\tau_{\cl_{\cal{S}}}(R)=\bigcap_{S \in \cal{S}} \intr_S(R).\]
We know that $\cl_{\cal{S}}$ is trivial if and only if $\tau_{\cl_{\cal{S}}}=R$. The right hand side is equal to $R$ if and only if $\intr_S(R)=R$ for all $S \in \cal{S}$, which holds if and only if each $S$ has a free summand (Lemma \ref{properties interior}, part (6)).

Alternatively, this follows from Definition \ref{familytestideal} and Corollary \ref{cor: test ideal trivial = direct summand}.
\end{proof}

Ideally, we want to consider the test ideal coming from the family of all Cohen-Macaulay modules, since a ring is regular if and only if the test ideals of these modules are equal to the whole ring by Corollary \ref{testidealtrivialonregular}. The collection of \CM\ modules is not generally a set, so we work with the following family instead:

\begin{remark}
\label{basremark}
Let $R$ be any Cohen-Macaulay ring and consider the full subcategory of $Mod(R)$ consisting of big \CM\ modules over $R$. For any set $S$ the module $R^S$ is in  this subcategory, hence there is an embedding of the category of sets to the category of \CM\ modules over $R$. The former is not a small category, so the latter is not a small category either. 

To avoid this complication we restrict to a single representative for each isomorphism class of \CM\ modules and bound the size of the modules we consider. 
To do this, let $R$ be a local ring and $Bas$ be a fixed infinite set. Let $CM(R)$ be the full subcategory of big \CM\ $R$-modules that are quotients of free $R$-modules $R^S$ with $S \subseteq Bas$. This is a small category, and therefore we can consider the set of objects in this category. 
For the purposes of this paper, it is enough for $Bas$ to have countable order, and 
we denote the set of objects by $CM$.
Since isomorphic modules give the same closure operation, test ideal, and trace ideal, studying $CM$ is sufficient for our purposes.
\end{remark} 

\begin{definition} 
\label{tausingdef}
Let $R$ be a complete local domain. We define the singular test ideal to be \[\tau_{sing}(R)=\bigcap_{B\in CM} \tau_B(R),\]  
where $CM$ is defined as in Remark \ref{basremark}.
\end{definition}

\begin{proposition}
Let $(R,m,k)$ be a complete local domain, then $R$ is regular if and only if $\tau_{sing}(R)=R$.
\end{proposition}

\begin{proof}
If $R$ is regular, by Corollary \ref{testidealtrivialonregular}, $\tau_B(R)=R$ for all big \CM\ $R$-modules $B$. Hence $\tau_{sing}(R)=R$.

If $\tau_{sing}(R)=R$, then $\tau_B(R)=R$ for all countably-generated big \CM\ $R$-modules $B$. Hence for such $B$, $\cl_B$ is trivial on all submodules of all $R$-modules. Let $\cl$ be a Dietz closure on \fg\ $R$-modules. By Theorem 5.1 of \cite{RGRebecca_Closures_Big_CM_modules_and_singularities} there exists a countably-generated big \CM\ $R$-module $B$ such that $\cl \subseteq \cl_B$ on submodules of \fg\ $R$-modules. Note that $B$ is not explicitly described as countably-generated in \cite{RGRebecca_Closures_Big_CM_modules_and_singularities}, but the process of constructing $B$ using module modifications uses countably many steps, each adding a finite number of generators. Hence $\cl$ is trivial on submodules of \fg\ $R$-modules. Since this holds for all Dietz closures $\cl$, $R$ is regular by Theorem \ref{trivialiffregular}.
\end{proof}

The following results connect the test ideals of big \CM\ modules to the singular locus of the ring, and are used to get more specific results on test ideals of big \CM\ modules over rings with finite \CM\ type.

\begin{remark}
In this paper, all \fg\ \CM\ $R$-modules are assumed to be maximal, i.e. to have dimension equal to $\dim R$. 
\end{remark}

\begin{theorem} {\label{Prop:Singular contains V(int)}}
Let $B$ be a  \fg\  \CM\ module over a local domain $R$. Then $V(\intr_B(R))$ is contained in the singular locus of $R$. 
\end{theorem}

\begin{proof}

Suppose otherwise, then there exists $P\in \Spec(R)$ such that $R_P$ is a regular ring and $\intr_B(R)\subseteq P$. After localizing at $P$ this implies $\intr_{B_P}(R_P) \subseteq PR_P$. Since $B$ is faithful over $R$, $B_P$ is nonzero. It is also finitely-generated, so by Nakayama's lemma $PB_P \ne B_P$. Now, $B_P$ is a \CM\ module over the regular local ring $R_P$, hence faithfully flat over $R_P$ \cite[Pag. 77]{HochsterHunekeInfiniteIntegralExtensions}, a local ring, and hence $\tau_{B_P}(R_P)=R_P$ ($B_P$ gives the trivial closure, so it gives the whole ring as the test ideal). This implies that
$\intr_{B_P}(R_P)=R_P$, a contradiction.
\end{proof}

This leads to a statement for test ideals.

\begin{corollary} {\label{Prop:Singular contains V(tau)}}
Let $B$ be a \fg\ \CM\  module over a complete local domain $R$. Then $V(\tau_B(R))$ is contained in the singular locus of $R$. 
\end{corollary}

\begin{proof}
This follows immediately from the previous result and  Theorem \ref{thm: test = int}. 
\end{proof}

\begin{remark}
We denote by $MCM(R)$ the set of all \fg\ (maximal) \CM\ modules over $R$.
We will write just $MCM$ if $R$ is understood from the context.
\end{remark}

\begin{definition}
Let $R$ be a local ring. $R$ has \textit{finite \CM\ type} if $R$ has finitely many isomorphism classes of indecomposable \fg\ \CM\ modules.
\end{definition}

If $R$ is a local ring of finite \CM\ type, we know the following:

\begin{itemize}
\item (Auslander \cite[Theorem 7.12]{LeuschkeWiegandCohenMacaulayRepresentations}) $R$ has isolated singularities.
\item If $R$ is not regular then the top dimensional syzygy $S$ of the residue field $k$ is a \fg\ \CM\ module for $R$ with no free summand \cite[Corollary 1.2]{SyzygiesHomologicalConjecturesDutta}. Hence by Proposition \ref{properties interior} part 6, $\intr_S(R) \neq R$ and by Corollary \ref{cor: test ideal trivial = direct summand}, $\tau_S(R) \ne R$.
\end{itemize}

\begin{proposition}
\label{finitecmtypemprimary}
Suppose that $(R,m)$ is a \CM\ ring with finite \CM\ type. If $R$ is not regular then \[\sqrt{\bigcap_{M\in \text{MCM}}\intr_M(R)} = m.\] Consequently \[\sqrt{\tau_{ \text{MCM}}(R)} = m.\]
\end{proposition}

\begin{proof}
Let $M$ be a \fg\  \CM\ module over $R$. Then by Proposition \ref{Prop:Singular contains V(int)}, since $R$ has an isolated singularity, $\sqrt{\intr_M(R)}$ is either $m$-primary or $R$. From the facts above there is at least one MCM module (say the top dimensional syzygy) that gives an $m$-primary trace ideal. Since a finite intersection of $m$-primary ideals is $m$-primary, the result follows. 
\end{proof}

The following results connect the trace ideal, and hence the test ideal, to the socle of the ring (the set of elements annihilated by the maximal ideal $m$). Rings with nonzero socle are not reduced.

\begin{lemma}
\label{interiorcontainssocle}
Let $(R,m)$ be a local ring and $B$ an $R$-module such that $B/mB$ is nonzero (for example, $B$ could be a nonzero \fg\ module). Then $\text{soc}(R) \subseteq \intr_B(R)$.
\end{lemma}

\begin{proof}
Since $B \ne mB$,  $B/m B$ is a nontrivial $R/m$-vector space, so we can find a surjective morphism from $B/m B$ to $R/m$. In particular we have a surjection $B \to R/m$. If $x$ is an element of the socle of $R$, then there is a map from $B\to R/m \to R$ that first sends $B$ onto $R/m$ and then to $R/(xm) \cong R$ via multiplication by $x$. Some element of $B$ maps to $1$ in $R/m$, and this maps to $x$ in $R$. From this we see that $\text{soc}(R) \subseteq \intr_B(R)$. 
\end{proof}

\begin{corollary}
Let $R$ be a local ring and $B$ an $R$-module such that $B/mB$ is nonzero. If $B$ is finitely-presented or $R$ is complete then $\text{soc}(R) \subseteq \tau_B(R)$.
\end{corollary}

\begin{proof}
This follows from Theorem \ref{thm: test = int} and Lemma \ref{interiorcontainssocle}.
\end{proof}

As a consequence of these results, when $R$ is zero-dimensional, we can say exactly what the singular test ideal is.

\begin{theorem}
Let $(R,m)$ be an Artinian local ring. Then $\bigcap_{B \in CM} \intr_B(R)$ is nonzero. In fact, \[\bigcap_{B \in CM} \intr_B(R) = \text{soc}(R)=\tau_{sing}(R).\] 
Hence if $m \ne 0$, then $\bigcap_{B \in CM} \intr_B(R) \ne 0.$
\end{theorem}

\begin{proof}
By Lemma \ref{interiorcontainssocle}, we know that for each $B \in CM$, $\intr_B(R) \supseteq \text{soc}(R)$. Hence 
\[\bigcap_{B \in CM} \intr_B(R) \supseteq \text{soc}(R).\]

For the other inclusion, note that since $R$ is zero dimensional, $k=R/m$ is a \CM\ module. The image of any map from $k$ to $R$ lives in $\text{soc}(R)$. So $\intr_k(R) \subseteq \text{soc}(R)$. Hence
\[\bigcap_{B \in CM} \intr_B(R) \subseteq \intr_k(R) \subseteq \text{soc}(R).\]

The second equality follows from Corollary \ref{testidealoffamilyintermsofint}.
\end{proof}

In the one-dimensional case, we prove that $\tau_{MCM}(R)=\intr_{MCM}(R) \ne 0$ under the hypothesis that $R$ is analytically unramified (i.e., its completion is reduced).
We use several definitions from \cite[Chapter 4]{LeuschkeWiegandCohenMacaulayRepresentations} .

\begin{definition}
Let $R$ be a domain of dimension one (so $R$ is \CM), let $K$ be the fraction field of $R$, and let $\bar{R}$ be the integral closure of $R$ in $K$. The conductor $ \mathfrak{c} = (R:_R \bar{R})$ is the largest common ideal of $R$ and $\bar{R}$, and is nonzero.
\end{definition}

If $M$ is a \fg\  \CM\ $R$-module, then $M$ is torsion-free. We use $\bar{R}M$ to denote the $\bar{R}$-submodule of $K \otimes_R M$ generated by  $\IM(M \to K \otimes_R M)$.  
This module is $\bar{R}$-projective \cite[Chapter 4]{LeuschkeWiegandCohenMacaulayRepresentations}. 

\begin{proposition}
Let $R$ be a local domain of dimension 1 (hence \CM). Then for any \fg\ $R$-module $M$ we have $\intr_M(R)\supseteq (R:\bar{R})$. This implies that $\intr_{MCM}(R) \supseteq (R:\bar{R})$. If $R$ is analytically unramified (in particular if $R$ is complete), then $\intr_M(R)$ and $\intr_{MCM}(R) \neq 0$.
\end{proposition}

\begin{proof}
Let $M$ be a \fg\ \CM\ module over $R$. Then $\bar{R}M$ is a projective module over the regular ring $\bar{R}$. It follows that there is a surjective $\bar{R}$-linear map $\phi:\bar{R}M\to \bar{R}$. In particular, there exist $\bar{r}_j \in \bar{R}$ and $m_j \in M$ such that $\phi\left(\sum_{j=1}^n \bar{r}_jm_j\right)=1$.
Therefore for $c \in (R:\bar{R})$, the map $c \cdot \phi$ sends $ \sum_j (c\bar{r}_jm_j) \mapsto c$ and has image in $R$, so we conclude that $  (R:\bar{R}) \subseteq \intr_M(R)$.

The statement that $\intr_{MCM}(R) \supseteq (R:\bar{R})$ follows immediately.

For the last part of the result, note that if $R$ is analytically unramified, then $\bar{R}$ is module-finite over $R$, which implies that $(R:\bar{R})$ is nonzero \cite[Chapter 4]{LeuschkeWiegandCohenMacaulayRepresentations}
\end{proof}

We now discuss the test ideal given by the family of big \CM\ $R$-algebras. The following result of Hochster indicates that tight closure on \fg\ $R$-modules comes from big \CM\ $R$-algebras. Our study of the test ideal coming from the family of big \CM\ algebras is motivated by the view that big \CM\ algebras are a useful tight closure replacement in all characteristics.

\begin{theorem}[{\cite[Theorem 11.1]{solidclosure}}]
\label{tightclosurebigcmalgclosure}
Let $R$ be a complete local domain of \charp, and let $N \subseteq M$ be \fg\ $R$-modules. Then $N_M^*$, the tight closure of $N$ in $M$, is equal to the set of elements $u \in M$ that are in $N_M^{\cl_B}$ for some big \CM\ algebra $B$.
\end{theorem}

\begin{definition}
Let $CMA$ be the subcategory of \CM\ $R$-algebras with basis elements obtained from \textit{Bas} as described in Remark \ref{basremark}. We define 
\[\tau_{CMA}(R) = \bigcap_{B \in CMA} \tau_B(R).\] 
We can also define the finitistic version, 
\[\tau_{CMA}^{fg}=\bigcap_{B \in CMA} \tau_{B}^{fg}(R).\]
\end{definition}

The following result indicates why big \CM\ algebra test ideals are a good tight closure replacement.

\begin{theorem}
\label{taufgequalsfinitistictestideal}
Let $R$ be a complete local domain of \charp. Then $\tau_{CMA}^{fg}(R)$ as defined above is equal to the finitistic tight closure test ideal.
\end{theorem}

\begin{proof}
By Theorem \ref{tightclosurebigcmalgclosure}, for \fg\ $R$-modules $N \subseteq M$, $N_M^{\cl_B} \subseteq N_M^*$ for every big \CM\ algebra $B$. Hence for each $B \in CMA$, $\tau_{\cl_B}^{fg}(R) \supseteq \tau_*^{fg}(R)$. This implies that
\[\bigcap_{B \in CMA} \tau_{B}^{fg}(R) \supseteq \tau_*^{fg}(R).\]
For the other direction, note that for each \fg\ $R$-module $M$, there exist 
\[B_{1,M},\ldots,B_{n_M,M} \in CMA\] such that $0_M^* \subseteq 0_M^{\cl_{B_{1,M}}} + \ldots + 0_M^{\cl_{B_{n,M}}}.$  Hence 
\[\bigcap_{i=1}^{n_M} \Ann_R 0_M^{\cl_{B_{i,M}}} \subseteq \Ann_R 0_M^*.\]
This implies that
\[\bigcap_{M \text{ f.g.}} \bigcap_{i=1}^{n_M} \Ann_R 0_M^{\cl_{B_{i,M}}} \subseteq \bigcap_{M \text{ f.g.}} \Ann_R 0_M^*.\]
But the left hand side contains
\[\bigcap_{M \text{ f.g.}} \bigcap_{B \in CMA} \Ann_R 0_M^{\cl_B}=\bigcap_{B \in CMA} \bigcap_{M \text{ f.g.}} \Ann_R 0_M^{\cl_B}=\bigcap_{B \in CMA} \tau_B^{fg}(R),\]
and the right hand side is equal to $\tau_*^{fg}(R)$. Hence
\[\bigcap_{B \in CMA} \tau_B^{fg}(R) \subseteq \tau_*^{fg}(R),\]
which gives us equality.
\end{proof}

This result only concerns the finitistic test ideal because it is unknown whether tight closure and big \CM\ algebras give the same closure operation on all $R$-modules, or even the same big test ideal. We are still able to get the following consequence:

\begin{corollary}
\label{bigcmalgebratestideals}
Let $R$ be a complete local domain of \charp. Then $R$ is weakly F-regular (all \fg\ $R$-modules are tightly closed) if and only if $\tau_{B}^{fg}(R)=R$ for all big \CM\ algebras $B$.
\end{corollary}

If $R$ is a complete local domain of equal characteristic, Dietz and R.G. \cite{dietzseeds,dietzrg} construct a directed family of big \CM\ algebras, i.e., a family of big \CM\ $R$-algebras such that given big \CM\ algebras $B$ and $B'$, there is a big \CM\ algebra $C$ and $R$-algebra maps $B,B' \to C$ that give rise to the following commutative diagram, where the maps $R \to B$ and $R \to B'$ send $1 \mapsto 1$:
\[
\begin{CD}
B @>>> C \\
@AAA @AAA \\
R @>>> B' \\
\end{CD}
\]
 In \charp, this includes all big \CM\ $R$-algebras; in \echarz, this includes all big \CM\ $R$-algebras that are ultrarings. In these cases, we use the closure operation given by the family of big \CM\ $R$-algebras to define the test ideal.

\begin{definition}
Let $W$ be an infinite set with a non-principal ultrafilter $\mathcal{W}.$ For each $w \in W$, take a ring $A_w$. The ultraproduct $A_\natural$ of the $A_w$ (with respect to $\mathcal{W}$) is the quotient $(\Pi_w A_w)/I_{null}$, where $I_{null}$ is the ideal of elements $(x_w)_{w \in W}$ of $\Pi_w A_w$ where $x_w=0$ for all $w$ in some subset $V$ of $W$ contained in $\mathcal{W}$. Any such ring $A_\natural$ is called an ultraring.
\end{definition}

For our purposes, we will be dealing with rings of \echarz\ that are ultraproducts of rings of \charp, as in \cite{dietzrg}.

\begin{theorem} {\cite[Theorem 8.4]{dietzseeds}}
Let $R$ be a complete local domain of positive characteristic. If $B$ and $B'$ are big \CM\ $R$-algebras, then there is an $R$-algebra map $B \otimes B' \to C$ for some big \CM\ algebra $C$.
\end{theorem}

\begin{theorem} {\cite[Theorem 3.3]{dietzrg}}
Let $R$ be a local domain of equal characteristic zero, and $B$ and $B'$ big \CM\ $R$-algebras that are also ultrarings (ultraproducts of char $p$ approximations $R_w$ of $R$). Then there is a big \CM\ $R$-algebra $C$ and an $R$-algebra map $B \otimes B' \to C$.
\end{theorem}

In either case, we can define the test ideal of the directed family as in Definition \ref{familytestideal}.

\begin{corollary}
Let $R$ be a complete local domain of equal characteristic and let $\cal{B}$ be either the set of all big \CM\ $R$-algebras (if $R$ has \charp) or the set of big \CM\ $R$-algebras that are also ultrarings (if $R$ has \echarz), in both cases following the setup of Remark \ref{basremark} to ensure we get a set. Then $\tau_{\cal{B}}(R)$ is equal to the test ideal of the closure $\cl_{\cal{B}}$.
\end{corollary}

\section{Examples}
\label{examples}

In this section we compute test ideals and trace ideals. In these examples, we compute $\Hom_R(B,R)$ for various \CM\ modules $B$, and look at the images of these maps in $R$. In the situation of Theorem \ref{thm: test = int}, this gives us the test ideal $\tau_B(R)$, and in general it gives us the trace ideal $\intr_B(R)$.

\begin{example} 
\label{pidexample}
Let $R$ be a complete PID. Then for any family of $R$-modules $\mathcal{F}$ we either have $\intr_{\mathcal{F}}(R)=0$ or $\intr_{\mathcal{F}}(R)=R$. Indeed, if $\intr_{\mathcal{F}}(R)\neq 0$ then it is a principal ideal $I$. Let $I\to R$ be an isomorphism. Composing this isomorphism with the elements of $\Hom_R(\mathcal{F},R)$, whose images add up to all of $I$, we have for each element of $R$ a map from $\mathcal{F} \to R$ whose image includes that element. 
Hence $\intr_{\mathcal{F}}(R)=R$.

If $R$ is also local and $B$ is any big \CM\ $R$-module, $B$ is solid (i.e. $\Hom_R(B,R) \ne 0$), so $\intr_B(R)=R$. Hence $\tau_{sing}(R)=R$.
\end{example}

But this is not always true in the general one-dimensional case, as the following example shows.

\begin{example}
Let $R=k[[t^2,t^3]]$ where $k$ is a field. Let $B=\left< (t^4,t^3),(t^3,t^2) \right>\subset R^2$. This is a \fg\ \CM\ $R$-module. There is no surjective map $B\to R$. Indeed, if there were then there would be $a,b\in R$ such that $(at^4+bt^3,at^3+bt^2)\mapsto 1$. But note that if 
\[e=at^2(t^3,t^2)+b(t^4,t^3) = (at^5+bt^4,at^4+bt^3) \in B\] maps to  $x\in R$, we also have   
\[t^3(at^4+bt^3,at^3+bt^2)\mapsto t^3,\] but  \[t^3(at^4+bt^3,at^3+bt^2)=t^2(at^5+bt^4,at^4+bt^3)=t^2e\mapsto t^2 x.\] This implies that $t^2x=t^3$. However there is no element of $R$ that satisfies this equation.

Now consider the map $B\to R$ given by $(c,d)\mapsto d$. The image of this map is the ideal $\m=(t^2,t^3)$. Hence we can conclude that \[\intr_B(R)=\m.\]
\end{example}

\begin{example}\label{example quadratic surface}
Let $R=\frac{k[[a,b,c]]}{(b^2-ac)}=k[[x^2,xy,y^2]]$, where $k$ is a field. By  \cite[Proposition 1.16]{CohenMacaulayModulesOverCohenMacauayRingsYoshino}  high syzygies ($\dim(R)$ or higher) $K$ of the residue field $k$ are \CM\ modules if they are nonzero, and by Remark \ref{Rmk: Top syz induce nontrivial closures} $\cl=\cl_K$ is non-trivial. Hence $\tau_K(R) \ne R$. Using Macaulay2 we find that the free resolution for the residue field has the form \[R^4\to R^3 \to R \to R/\m \to 0 \]  where the map $R^4\to R^3$ is given by the matrix 
\[
\begin{pmatrix}
-y^2 & -xy & 0 & -y^2 \\
xy & x^2 & -y^2 & 0 \\
0 & 0 & xy & x^2 \\
\end{pmatrix}
\]
Hence $K=\text{syz}_2(k)$ is the $R$-submodule of $R^3$ generated by the columns of this matrix. Let $I=(x^2,y^2)$. Then $\text{rad}(I)=m$. We claim that $xy \in I^{\cl_K}$. Since $I$ is an ideal, $I^{\cl_K}=(IK:K)$. Hence it is enough to show that $xyK \subseteq IK$. Multiplying $xy$ by each of the columns of the matrix above, we have

\[
xy \begin{pmatrix}
-y^2 \\ xy \\ 0 \\
\end{pmatrix}
= \begin{pmatrix}
-xy^3 \\ x^2y^2 \\ 0 \\
\end{pmatrix}
= y^2\begin{pmatrix}
-xy \\ x^2 \\ 0 \\
\end{pmatrix},
\]

\[xy \begin{pmatrix}
-xy \\ x^2 \\ 0 \\
\end{pmatrix}
= \begin{pmatrix}
-x^2y^2 \\ x^3y \\ 0 \\
\end{pmatrix}
= x^2\begin{pmatrix}
-y^2 \\ xy \\ 0 \\
\end{pmatrix},
\]

\[
xy \begin{pmatrix}
0 \\ -y^2 \\ xy \\
\end{pmatrix}
= \begin{pmatrix}
0 \\ -xy^3 \\ x^2y^2 \\
\end{pmatrix}
= -y^2\begin{pmatrix}
-y^2 \\ xy \\ 0 \\
\end{pmatrix}
+y^2\begin{pmatrix}
-y^2 \\ 0 \\ x^2 \\
\end{pmatrix},
\]

\[
xy \begin{pmatrix}
-y^2 \\ 0 \\ x^2 \\
\end{pmatrix}
= \begin{pmatrix}
-xy^3 \\ 0 \\ x^3y \\
\end{pmatrix}
= y^2\begin{pmatrix}
-xy \\ x^2 \\ 0 \\
\end{pmatrix}
+x^2\begin{pmatrix}
0 \\ -y^2 \\ xy \\
\end{pmatrix},
\]
which implies that $xyK \subseteq IK$. Hence $I^{\cl_K}=m$, and so $I:I^{\cl_K}=m$. Therefore,
\[ \tau_{sing}(R) \subseteq \tau_K(R)\subseteq \m. \].
\end{example}

\begin{example}
By an alternate method, we can say exactly what $\tau_{sing}(R)$ is in this case.
Let $R=k[[x^2,xy,y^2]] \subseteq k[[x,y]]=S$, where $k$ is a field. Then $R$ has exactly two indecomposable \fg\ \CM\ modules, $R$ and $M=xR+yR \subseteq S$. By a result of \cite{hochsterleuschkerg}, if $B$ is a big \CM\ module over $R$, then either $R$ or $M$ splits from $B$. Since for any modules $A$ and $N$, $\cl_{A \oplus N}=\cl_{A} \cap \cl_{N}$, this means that $\cl_M$ gives the largest big \CM\ module closure on $R$. So $\tau_M(R)=\tau_{sing}(R)$.

Since $M \cong (x^2,xy)R \cong (xy,y^2)R$, $\tau_{sing}(R)=\tau_M(R)$ must contain $m=(x^2,xy,y^2)R$. However, since $R$ is not regular, $\tau_{sing}(R) \ne R$. Therefore, $\tau_{sing}(R)=m$.
\end{example}

The following example is of a ring with countable \CM\ type whose singular test ideal is not primary to the maximal ideal. This indicates that Proposition \ref{finitecmtypemprimary} does not hold even for fairly nice rings with infinite \CM\ type.

\begin{example}
\label{whitneyumbrella}
Let $R=k[[x,y,z]]/(x^2y+z^2)$, where $k$ is a field of arbitrary characteristic. This ring is known as the $D_\infty$ hypersurface singularity and as the Whitney Umbrella. By \cite[Proposition 14.19]{LeuschkeWiegandCohenMacaulayRepresentations}, this ring has countable \CM\ type and  the isomorphism classes of idecomposable, non-free \fg\ \CM\ modules are obtained as the cokernels of each of the following matrices. 

\begin{itemize}

\item $M=\operatorname{coker} \begin{pmatrix}
z & y \\ 
-x^2 & z
\end{pmatrix},
$

\item $N=\operatorname{coker} \begin{pmatrix}
z & xy \\ 
-x & z
\end{pmatrix}
$

\item $M_j=\operatorname{coker} \begin{pmatrix}
z & 0 & xy & 0 \\ 
0 & z & y^{j+1} & -xy \\
-x & 0 & z & 0 \\
-y^j & x & 0 & z
\end{pmatrix}
$

\item $N_j=\operatorname{coker} \begin{pmatrix}
z & 0 & xy & 0 \\ 
0 & z & y^{j} & -x \\
-x & 0 & z & 0 \\
-y^{j} & xy & 0 & z
\end{pmatrix}
$
\end{itemize}

Let's compute the corresponding test ideals. As the ring $R$ is a complete local domain, by Theorem \ref{thm: test = int} we only need to compute the trace ideal of $R$ with respect to these modules. 

\begin{itemize}

\item $M$ : A map $\phi$ from $M$ to $R$ is the same as a map from $R^2 \to R$ whose kernel contains $<(z,x^2),(-y,z)>$. That is, we must have that $z\phi(e_1)+x^2\phi(e_2)=-y\phi(e_1)+z\phi(e_2)=0$, or in an equivalent way, we want solutions for \[ \begin{pmatrix}
z & -x^2 \\ 
y & z
\end{pmatrix} \begin{pmatrix}
a \\ b
\end{pmatrix} = 0
\]
with $a,b \in R$. We first find the solutions in the fraction field and then determine when they are in $R$. To do this, we row reduce this matrix by multiplying the second row by $x^2$ and then adding the $z$ times the first row, which gives us 
\[ \begin{pmatrix}
z & -x^2 \\ 
0 & 0
\end{pmatrix} \begin{pmatrix}
a \\ b
\end{pmatrix} = 0.
\]
This means that we need $az=bx^2$.  As we want $a,b\in R$, this is equivalent to saying $a \in x^2:z$ and $b \in z:x^2$. It follows that $\tau_M(R)=\intr_M(R)=(x^2:z)+(z:x^2)$. As both ideals are proper, $\intr_M(R)\neq R$. Now, note that from the equation $z^2=-x^2y$ we have that $(x^2,y,z) = \intr_M(R)$.

\item N: A similar procedure implies $\tau_N(R)=\intr_N(R)=(z:x)+(x:z)$, which is equal to $(xy,z)+(z,x)=(x,z)$.

\item $M_j$: After transposing and row reducing we obtain the system 
\begin{align*}
az-cx-dy^j &=0\\
bz+dx &= 0.
\end{align*}
Some possibilities that satisfy this equation are (found in Macaulay2 for particular values of $j$, but easy to check that they are correct for any $j$):
\[
\begin{array}{c | c | c | c}
a & b & c & d \\
\hline
\hline
x & 0 & z & 0 \\
\hline
(-1)^jy^j & x & 0 & (-1)^jz \\
\hline
-z & 0 & xy & 0 \\
\hline
0 & x & -y^{j+1} & xy \\
\end{array}
\]
Hence $(x,y^j,z) \subseteq \tau_{M_j}(R).$ (In fact, computations in Macaulay2 confirm that these choices generate all maps $M_j \to R$, so the two ideals are equal.)

\item $N_j$: As in the previous case, transposing and row reducing we obtain the system 

\begin{align*}
az-cx-dy^j &=0\\
bz+dxy &= 0.
\end{align*}

In particular the following are solutions to this set of equations
\[
\begin{array}{c | c | c | c}
a & b & c & d \\
\hline
\hline
x & 0 & z & 0 \\
\hline
y^j & -xy & 0 & z \\
\hline
-z & 0 & xy & 0 \\
\hline
0 & z & -y^j & x \\
\end{array}
\]
so $(x,y^j,z) \subseteq \tau_{N_j}(R)$. (As with $M_j$, Macaulay2 computations confirm that they are actually equal.)
\end{itemize} 

From this we can conclude that the intersection of $\tau_B(R)$ over all \fg\ \CM\ $R$-modules $B$ is
\[(x^2,y,z) \cap (x,z) \cap (x,y^j,z) \cap (x,y^j,z)=(x^2,xy,z).\]
Notice that this is not primary to the maximal ideal, and so the singular test ideal, which is contained in this ideal, is also not $m$-primary.
\end{example}

Even though we have only defined test ideals for domains, we can compute trace ideals without this hypothesis. In the next example we compute the trace ideal of a non-domain ring with respect to its \fg\ \CM\ modules.

\begin{example}
Let $R=k[x,y,z]/(xz)$, where $k$ is an algebraically closed field of characteristic not equal to 2. 
We will use $i$ to denote $\sqrt{-1}$. In this case $R$ has countably infinite \CM\ type, that is, up to isomorphism, there are countably many indecomposable \fg\ \CM\ $R$-modules. By isomorphism with $k[x,y,z]/(x^2+z^2)$ via
\begin{align*}
x &\mapsto z-ix \\
y &\mapsto y \\
z &\mapsto z+ix,
\end{align*}
we see that this is the same as the example in \cite[Proposition 14.17]{LeuschkeWiegandCohenMacaulayRepresentations}. 
Hence the indecomposable \fg\ \CM\ $R$-modules are given as the cokernels $M_j$ of $\phi_j:R^2\to R^2$, where 
\[ \phi_j =\begin{pmatrix}
z & -y^j \\
0& x
\end{pmatrix}\]
and the cokernels $M_j'$ of $\phi_j':R^2 \to R^2$, where
\[ \phi_j' =\begin{pmatrix}
x & y^j \\
0& z
\end{pmatrix}\]
or as the cokernel, $M$ of $\psi=(x)$ or $M'$ of $\psi=(z)$.

We claim that $\intr_{M_j}(R)=\intr_{M_j'}(R)=(x,y^j,z)$ and that $\intr_M(R)=\intr_{M'}(R)=(x,z)$. 
A map from $M_j \to R$ must send its natural generators to elements $a,b$ satisfying the relations 
\begin{align*}
az &=0\\
-ay^j+bx&=0.
\end{align*}
The first implies that $a=fx$ for some $f \in R$, and so $x(b-fy^j)=0$. This, in turn, implies \[ b= fy^j + gz\] for some $g \in R$. This implies that $a,b \in (x,z,y^j)$.  Now, choosing $f=1$ and $g=0$ gives the solution $a=x$ and $b=y^j$. This implies that $\intr_{M_j}(R)\supseteq (x,y^j)$. Similarly, choosing $f=0$ and $g=1$ gives $\intr_{M_j}(R) \supseteq (z)$, hence $\intr_{M_j}(R)=(x,z,y^j)$. The case of the $M'_j$ is similar.

However, $\intr_M(R)=(z)$ and $\intr_{M'}(R)=(x)$.

This implies that 
\[
\bigcap_{B} \intr_B(R) = (x) \cap (z)=(0),
\]
where the intersection is taken over all \fg\ indecomposable \CM\ $R$-modules.
\end{example}

\begin{remark}
Here $\tau_{MCM}(R)=0$. This supports the need for the domain hypothesis in many of the results of this paper.
\end{remark}

One more example of modules for which we can say something is the following
\begin{example} 

Let $R=k[x,y,z]/(z^2)$ localized at $(x,y,z)$ and set \[M_n= \bigoplus_{i=1}^{2n} k[x,y].\]

We make $M_n$ an $R$-module via 
\[z=\begin{pmatrix}
0 & \Phi \\
0 & 0 
\end{pmatrix}, \] where $\Phi$ is the $n\times n$ matrix \[\Phi=\begin{pmatrix}
x & y & 0 &  0 & \cdots & 0 \\
0 & x & y & 0 & \cdots & 0 \\
0 & 0 & x & y & \cdots & 0 \\
\vdots & & & & & \\
0 & 0 & 0 & 0 & \cdots & y \\
0 & 0 & 0 & 0 & \cdots & x 
\end{pmatrix} \]

By Proposition 3.4 of \cite{LeuschkeWiegandCohenMacaulayRepresentations}, $M_n$ is an indecomposable \CM\ module over $R$ for all $n \ge 2$. We compute $\intr_{M_n}(R)$. Let $e_1,\ldots,e_{2n}$ be the obvious set of generators for $M_n$. For any map $\psi:M_n\to R$ we have that $\psi$ is determined by $\psi(e_i)$. Notice that $z$ has the following action on the $e_i$:
\[
ze_i = \begin{cases} 0 & i \le n \\ xe_1 & i=n+1 \\ ye_{i-n-1}+xe_{i-n} & n<i \le 2n \\ \end{cases}
\]
We have a map $\psi:M_n \to R$ sending
\begin{align*}
e_1 &\mapsto z \\
e_{n+1} &\mapsto x \\
e_{n+2} &\mapsto y \\
e_i &\mapsto 0 \text{ for all other $i$.} \\
\end{align*}
To see that this is an $R$-linear map, we check that the action of $z$ is compatible with the map. We have
$z\psi(e_{n+1})=zx$ and
\[\psi(ze_{n+1})=\psi(xe_1)=x\psi(e_1)=xz.
\]
Additionally, $z\psi(e_{n+2})=zy$ and
\[\psi(ze_{n+2})=\psi(ye_1+xe_2)=y\psi(e_1)+x\psi(e_2)=yz+0=yz,\]
and $z\psi(e_1)=z^2=0=\psi(0)=\psi(ze_1).$
For $1<i<n$, $\psi(ze_i)=\psi(0)=0=z\psi(e_i).$
For $i>n+2$, $ze_i$ is in terms of $e_j$ for $1<j \le n$, so $0=z\psi(e_i)=\psi(ze_i).$

The existence of this map shows that $(x,y,z) \subseteq \intr_{M_n}(R)$ for $n \ge 2$. Hence 
\[(x,y,z) \subseteq \bigcap_{n} \intr_{M_n}(R).\]
To see that these are in fact equal, suppose there is a map $\psi:M_n \to R$ sending $e_i \mapsto 1$ for some $i$. If $i \le n$, we have
\[z=z\psi(e_i)=\psi(ze_i)=0,\]
which is a contradiction.
If $i=n+1$, we have
\[z=z\psi(e_i)=\psi(ze_i)=\psi(xe_1)=x\psi(e_1),\]
which is also a contradiction as $z \not\in (x)R$.
If $i>n+1$, we have
\[z=z\psi(e_i)=\psi(ye_{n-i-1}+xe_{n-i})=y\psi(e_{n-i-1})+x\psi(e_{n-i}),\]
which is a contradiction since $z \not\in (x,y)R$.
Hence $1 \not\in \intr_{M_n}(R)$, which implies that 
\[\intr_{M_n}(R)=(x,y,z).\]
\end{example}

\begin{remark} 
\label{questionifandonlyif} 
Given Proposition \ref{generationandclosures} it is natural to ask whether  $\tau_T(R)\subseteq \tau_S(R)$ if and only if $S$ generates $T$. Note that the ``if'' part follows from Proposition \ref{generationandclosures}. But as Example \ref{Example: tau containment does not imply generation} shows the other direction is false, even in the case of \fg\ \CM\ modules.
\end{remark}

\begin{example} \label{Example: tau containment does not imply generation}
Let $R=k[[x,y,z]]/(xy-z^4)$ where $k$ has characteristic 0 (or most values of $p$ are also fine). We can view $R$ as a subring of $S=k[[s,t]]$, via $x \mapsto s^4$, $y \mapsto t^4$, and $z \mapsto st$. The indecomposable MCM's of $R$ are $R$, $M_1=(s,t^3) \cong (y,z)$, $M_2=(s^2,t^2) \cong (y,z^2)$, and $M_3=(s^3,t) \cong (x,z)$ \cite{LeuschkeWiegandCohenMacaulayRepresentations}. According to Macaulay2, $H=\Hom(M_1,R)=\IM\begin{pmatrix}y & z^3 \\ z & x \\ \end{pmatrix}$, and using the function $\text{homomorphism} (H_{\{i\}})$ for $i=0,1$, we see that the homomorphisms $M \to R$ are as follows: one of them is given by $s \mapsto y$ and $t^3 \mapsto z$, and the other by $s \mapsto z^3$ and $t^3 \mapsto x$.

Similarly, $\Hom(M_3,R)=\IM\begin{pmatrix}x & z^3 \\ z & y \\ \end{pmatrix}$, and the homomorphisms send $s^3 \mapsto x$, $t \mapsto z$, or $s^3 \mapsto z^3$, $t \mapsto y$.

So $\intr_{M_1}(R)=\intr_{M_3}(R)=m$. But $M_1$ and $M_3$ are distinct indecomposable \CM\ $R$-modules, so neither generates the other. As $M_1$ and $M_3$ are \fg\ $R$-modules, $\intr_{M_i}(R) = \tau_{M_i}(R)$ for $i=1,3$, so $M_1$ and $M_3$ are two $R$-modules that give the same test ideal, but neither one generates the other.
\end{example}

\section{Mixed Characteristic}
\label{mixedchar}

Recently, Andr\'{e} proved the existence of big \CM\ algebras in mixed characteristic \cite{andre}. 
We take advantage of this result and of almost big \CM\ algebras as defined by Roberts \cite{robertsalmostcm} and used by Andr\'{e} to define a closure operation in mixed characteristic, and to prove that the corresponding test ideal can be written as a variant on a trace ideal, paralleling our results in previous sections. This demonstrates that the arguments used in earlier sections can be adapted to apply to closures that are variations on module closures.

Our closure is similar to dagger closure as defined by Hochster and Huneke \cite{ElementsSmallOrderIntegralExtensionsHochsterHuneke}. The key difference is that we have replaced $R^+$, the absolute integral closure of $R$, with an arbitrary almost big \CM\ algebra. We are also using small powers of a particular element as our ``test elements", as is usual in working with perfectoid algebras, rather than using arbitrary elements of small order as in \cite{ElementsSmallOrderIntegralExtensionsHochsterHuneke}.

In this section, let $(R,m)$ be a complete local domain of dimension $d>0$ and mixed characteristic $(0,p)$, $T$ a $p$-torsion free algebra, and $\pi \in T$ a \nzd\ such that $T$ contains a compatible system of $p$-power roots of $\pi$, i.e. a set of elements $\{\pipn\}_{n \ge 1}$ such that $(\pipn)^{p^m}=\pip{n-m}$ for all $m \le n$. 
We will denote this system of $p$-power roots of $\pi$ by $\pipinfty$. 

\begin{definition}[{\cite[Definition 4.1.1.3]{andre}}]
$T$ is an {\it almost (balanced) big \CM\ algebra \wrt\ $\pipinfty$} if $T/mT$ is not almost 0 \wrt\ $\pipinfty$ (i.e., it is not the case that $\pipn T/mT=0$ for all $n>0$), and for every \sop\ $x_1,\ldots,x_d$ on $R$, 
\[\pipn \cdot \frac{(x_1,\ldots,x_i):_T x_{i+1}}{(x_1,\ldots,x_i)}=0\]
for all $n>0$, $0 \le i \le d-1$. 
\end{definition}

Andr\'{e} proved the existence of almost big \CM\ algebras as a step on the way to proving the existence of big \CM\ algebras. The reason we have included this ``intermediate" step in our paper (rather than focusing solely on big \CM\ algebras) is that almost mathematics is central to major results in mixed characteristic commutative algebra, and our techniques can be applied to this case. This also connects our results to the recent work of \cite{maschwede} on a mixed characteristic version of a test ideal for pairs in regular rings, which is defined using an almost big \CM\ algebra.

\begin{definition}
\label{mixedcharclosure}
Let $T$ be an almost big \CM\ algebra over $R$. We define a closure operation cl by $u \in N_M^{\cl}$ if for all $n>0$,
\[ \pipn \otimes u \in \IM(T \otimes_R N \to T \otimes_R M).\]
\end{definition}

\begin{proposition}
The closure cl defined above is a Dietz closure. Consequently, $\tau_{\cl}(R)=\Ann_R 0^{\cl}_{E_R(k)}$.
\end{proposition}
 
\begin{proof}
First, we show that cl gives a closure operation. Let $N \subseteq M$ be $R$-modules. It is clear that $N \subseteq N_M^{\cl}$. Additionally, if $N \subseteq N' \subseteq M$, and $u \in N_M^{\cl}$, then for all $n>0$, 
\[\pipn \otimes u \in \IM(T \otimes N \to T \otimes M) \subseteq \IM(T \otimes N' \to T \otimes M).\]
 Hence $N_M^{\cl} \subseteq (N')_M^{\cl}$. It remains to show that cl is idempotent. Suppose that $u \in (N_M^{\cl})_M^{\cl}$. Then for all $n>0$, 
 \[\pipn \otimes u \in \IM(T \otimes N_M^{\cl} \to T \otimes M).\] So we can write $\pipn \otimes u=\sum t_i \otimes m_i$, with the $m_i \in N_M^{\cl}$. Hence 
 \[\pip{n'} t_i \otimes m_i \in \IM(T \otimes N \to T \otimes M)\] for all $i$ and for all $n'>0$. This implies that for all $n,n'>0$, \[\pipn \pip{n'} \otimes u \in \IM(T \otimes N \to T \otimes M).\] In particular, $\pipn \pipn \otimes u \in \IM(T \otimes N \to T \otimes M)$ for all $n>0$. Multiplying by $\pi^{(p-2)/p^n}$, we get $\pip{n-1} \otimes u \in \IM(T \otimes N \to T \otimes M)$ for all $n>0$, so $u \in N_M^{\cl}$.

Next we prove that cl is functorial. Suppose that $f:M \to W$ is a map of $R$-modules, and $N \subseteq M$. Let $u \in N_M^{\cl}$. Then $\pipn \otimes u \in \IM(T \otimes N \to T \otimes M)$ for all $n>0$, i.e. $\pipn \otimes u = \sum t_i \otimes n_i$ with each $n_i \in N$. Applying $1 \otimes f$, we get $\pipn \otimes f(u)=\sum t_i \otimes f(n_i)$. Since each $f(n_i) \in f(N)$, we have $\pipn \otimes f(u) \in (f(N))_W^{\cl}$, as desired.

To prove semi-residuality, suppose that $N_M^{\cl}=N$. Let $u \in M$ such that $\bar{u} \in 0_{M/N}^{\cl}$. Then $\pipn \otimes \bar{u}=0$ in $T \otimes M/N$, which by right exactness of tensor products implies that $\pipn \otimes u \in \IM(T \otimes N \to T \otimes M)$. Hence $u \in N$, which implies that $\bar{u}=0$. Hence $0_{M/N}^{\cl}=0$.

For faithfulness, suppose that $u \in m_R^{\cl}$. Then $\pipn u \in mT$ for all $n>0$. If $u \not\in m$, then $u$ is a unit, so this implies that $\pipn \in mT$ for all $n>0$. But then $T/mT$ is almost zero, which is a contradiction.

For generalized colon-capturing, suppose $f:M \twoheadrightarrow R/I$, where $I=(x_1,\ldots,x_k)$ and $x_1,\ldots,x_{k+1}$ is part of a \sop\ for $R$, and let $v \in M$ such that $f(v)={\bar{x}}_{k+1}$. Let $u \in (Rv)_M^{\cl} \cap \Ker(f)$. Then 
\[\pipn \otimes u \in \IM(T \otimes Rv \to T \otimes M)\] for all $n>0$. So $\pipn \otimes u=t_n \otimes v$ for some $t_n \in T$. Hence 
\[0=(\text{id} \otimes f)(\pipn \otimes u)=(\text{id} \otimes f)(t_n \otimes v)=t_n \otimes f(v)\] in $T \otimes R/I$. So $t_nx_{k+1} \in IT$. Hence $\pip{n'}t_n \in IT$ for all $n'>0$. So 
\[\pipn \pip{n'} \otimes u \in \IM(IT \otimes Rv \to T \otimes M)=\IM(T \otimes Iv \to T \otimes M)\] for all $n,n'>0$. As in the proof of idempotence, this implies that $\pipn \otimes u \in \IM(T \otimes Iv \to T \otimes M)$ for all $n>0$. Therefore $u \in (Iv)_M^{\cl}$, which completes the proof of generalized colon-capturing.

As a corollary to Proposition \ref{prop: test = ann}, since cl is residual, $\tau_{\cl}(R)=\Ann_R 0^{\cl}_{E_R(k)}$.
\end{proof}

\begin{definition}
Let cl be the closure from Definition \ref{mixedcharclosure}. We define
\[\intr_{\cl}(R)=\sum_{n>0} \sum_{\psi:T \to R} \psi(\pipn T)=\sum_{n>0} \IM(T^* \otimes \pipn T \to R),\]
where $T^*=\Hom_R(T,R)$ and the map sends $h \otimes x \mapsto h(x)$.
\end{definition}

\begin{theorem} \label{thm: test = int mixed char}
Let $R$ be a complete local domain and let cl be the closure defined above. Then \[ \tau_{\cl}(R)=\intr_{\cl}(R).\]
\end{theorem}

\begin{proof}

Let $E=E_R(k)$ be the injective hull of the residue field $k$ of $R$. By Proposition \ref{prop: test = ann}, $\tau_{\cl}(R)=\Ann_R(0^{\cl}_E)=(0:0^{\cl}_E)$; hence $c\in \tau_{\cl}(R)$ if and only if $c\cdot 0^{\cl}_E=0$, but 
\[0^{\cl}_E=\bigcap_{n>0} \ker(E\xrightarrow{\alpha_n} T\otimes E),\] 
where $\alpha_n$ is given by $e \mapsto \pipn \otimes e$.
Since $E$ is Artinian, there are elements $n_1,\ldots,n_t > 0$ such that this is equal to 
\[\bigcap_{n \in \{n_1,\ldots,n_t\}} \ker(E \xrightarrow{\alpha_n} T \otimes E).\] 
We can rewrite this as $\ker(\phi)$, where $\phi=(\phi_1,\ldots,\phi_t):E \to (B \otimes E)^{\oplus t}$ sends 
\[e \mapsto (\pip{n_1} \otimes e,\pip{n_2} \otimes e,\ldots,\pip{n_t} \otimes e).\] 

First, suppose $c \in \tau_{\cl}(R)$, so that $c \cdot 0_E^{\cl}=c\ker(\phi)=0$. Then \[0^{\cl}_E\subseteq \Ann_E(c),\] and by Matlis duality the map  \[\frac{R}{cR}=\Hom_{R}(\Ann_E(c),E) \to \Hom_{R}(0^{\cl}_E,E)\] is surjective. But applying Matlis duality to the exact sequence \[0\to \ker(\phi)  \to E \xrightarrow{\phi} (T\otimes E)^{\oplus t} \] gives \[ \Hom_{R}(0^{\cl}_E,E)=\frac{R}{\sum_{n\in\{n_1,\ldots,n_t\}} \IM(\Hom_{R}(T,R)\to R)},\] where the maps $\Hom_{R}(T,R)\to R$ are given by $\psi\mapsto \psi(\pip{n_i})$ for each $n \in \{n_1,\ldots,n_t\}\}$. From the surjection $R/cR\to \Hom_{R}(0_E^{\cl},E)$ we can now conclude that 

\[c R \subseteq \sum_{n\in\{n_1,\ldots,n_t\}} \IM(\Hom_{R}(T,{R})\to R) .\] 

This gives us the desired result.

For the reverse containment, suppose that $c \in \intr_{\cl}(R)$. Then there are $n'_1,\ldots,n'_s>0$ such that $c \in \sum_{n \in \{n'_1,\ldots,n'_s\}} \IM(\Hom_R(T,R) \to R)$, where the $i$th map $\Hom_R(T,R) \to R$ sends $f \mapsto f(\pip{n'_i})$. We can enlarge the set $n_1,\ldots,n_t$ to include $n'_1,\ldots,n'_s$. Then $c \in \sum_{n \in \{n_1,\ldots,n_t\}} \IM(\Hom_R(T,R) \to R)$. Hence we have a surjection 
\[R/cR \to \frac{R}{\sum_{n \in \{n_1,\ldots,n_t\}} \IM(\Hom_R(T,R) \to R)}.\]
 Applying Matlis duality, we get an injection
\[ \Hom_R\left(\frac{R}{\left(\sum_{n \in \{n_1,\ldots,n_t\}} \IM(\Hom_R(T,R) \to R)\right)},E\right) \hookrightarrow \Hom_R(R/cR,E)=\Ann_E c.\] But the module on the left is $0_E^{\cl}$. Hence $c \in \Ann_R 0_E^{\cl}$, so $c \in \tau_{\cl}(R)$.
\end{proof}

\section{Acknowledgments}

The authors would like to thank Neil Epstein, Haydee Lindo, Keith Pardue, Karl Schwede, Kevin Tucker, Janet Vassilev, and Yongwei Yao for helpful conversations that improved this paper tremendously. In particular, Janet Vassilev shared information on interior operations, Karl Schwede suggested working on the mixed characteristic case, Kevin Tucker suggested Remark \ref{noncompletecounterexample}, Haydee Lindo taught the authors about trace ideals, Neil Epstein listened to some of the main proofs, Keith Pardue gave advice on Remark \ref{basremark} and Yongwei Yao discussed the proof of Proposition \ref{generationandclosures} with the second named author. The anonymous referee also made suggestions that improved the exposition of the paper.

\bibliographystyle{alpha}
\bibliography{References}


\vspace{.6cm}

{\small
\noindent  \textsc{Toronto, Ontario, Canada} \\ \indent \emph{Email address}: {\tt felipe@layer6.ai}

\vspace{.25cm}

\noindent \small \textsc{Department of Mathematics, George Mason University, Fairfax, VA 22030} \\ \indent  \emph{Email address}: {\tt rrebhuhn@gmu.edu}

\end{document}